\documentclass[]{amsart}

\usepackage{amssymb,amsthm,fancyheadings,amsmath,dsfont}

\newtheorem{defi}{Definition}[section]
\newtheorem{thm}[defi]{Theorem}
\newtheorem{lem}[defi]{Lemma}
\newtheorem{cor}[defi]{Corollary}
\newtheorem{pro}[defi]{Proposition}
\newtheorem{rem}[defi]{Remark}
\newcommand{\bpr}{\begin{proof}[Proof]}  
\newcommand{\epr}{\end{proof}}
\newcommand{\beq}{\begin{equation}}
\newcommand{\eeq}{\end{equation}}
\newcommand{\bce}{\begin{center}}
\newcommand{\ece}{\end{center}}
\newcommand{\be}{\begin{enumerate}}  
\newcommand{\ee}{\end{enumerate}}
\newcommand{\difft}{\frac{d}{dt}}

 \DeclareMathOperator*{\Real}{Re}

\DeclareMathOperator*{\dist}{dist}
\DeclareMathOperator*{\diver}{div}
\DeclareMathOperator*{\Mdiver}{Div}

\def\al{\alpha}
\def\pa{\partial}
\def\om{\omega}
\def\si{\sigma}

\def\la{\lambda}

\def\ep{\varepsilon}
\def\de{\delta}
\def\ph{\varphi}
\def\t{\tau}
\def\ka{\kappa}
\def\ga{\gamma}
\def\th{\theta}
\def\De{\Delta}
\def\Ga{\Gamma}

\def\Om{\Omega}

\def\R{\mathbb R}
\def\C{\mathbb C}

\def\N{\mathbb N}

\def\B{\mathbb B}
\def\E{\mathbb E}
\def\L{\mathbb L}

\def\calA{\mathcal A}
\def\calB{\mathcal B}
\def\calH{\mathcal H}
\def\calT{\mathcal T}

\def\calR{\mathcal R}

\def\calF{\mathcal F}
\def\calM{\mathcal M}
\def\calS{\mathcal S}

\def\calW{\mathcal W}

\def\calY{\mathcal Y}

\numberwithin{equation}{section}

\def\nn{\nonumber}

\title[On a Cahn-Hilliard-Gurtin system]
{$L_p$-Theory for a Cahn-Hilliard-Gurtin system}

\author[Mathias Wilke]{Mathias Wilke}

\address{Institut f\"{u}r Mathematik, Martin-Luther-Universit\"at Halle-Wittenberg,
06099 Halle, Germany}

\email{mathias.wilke@mathematik.uni-halle.de}


\date{\today}


\begin{document}

\maketitle

\begin{abstract}
In this paper we study a generalized Cahn-Hilliard equation which was proposed by Gurtin \cite{Gur}. We prove the existence and uniqueness of a local-in-time solution for a quasilinear version, that is, if the coefficients depend on the solution and its gradient. Moreover we show that local solutions to the corresponding semilinear problem exist globally as long as the physical potential satisfies certain growth conditions. Finally we study the long-time behaviour of the solutions and show that each solution converges to a equilibrium as time tends to infinity.
\end{abstract}

\section{Introduction}

We start with the derivation of the classical Cahn-Hilliard
equation. Consider the free energy functional of the form
\beq\label{freeEn}\calF(\psi)=\int_\Om\left(\frac{1}{2}|\nabla\psi|^2+\Phi(\psi)\right)\
dx,\eeq where $\Om$ is a bounded, open and connected subset of
$\R^n$ with boundary $\Ga:=\pa\Om\in C^3$. We assume that the order
parameter $\psi$ is a conserved quantity. The according conservation
law reads \beq\label{conslaw}\pa_t\psi+\diver j=0,\eeq where $j$ is
a vector field representing the phase flux of the order parameter.
The next step is to combine the two quantities $j$ and $\mu$.
Similar to Fourier's law in the derivation of the heat equation one
typically assumes that $j$ is given by
\beq\label{post}j=-\nabla\mu,\eeq a \emph{postulated} relation.
Finally we have to derive an equation for $\mu$. The chemical
potential $\mu$ is given by the \emph{variational derivative} of
$\calF$, i.e.
$$\mu=\frac{\de\calF}{\de\psi}=-\De\psi+\Phi'(\psi).$$
If $\calF$ is of the form \eqref{freeEn} this yields the classical
Cahn-Hilliard equation.

In the early nineties \textsc{Gurtin} \cite{Gur} proposed a
\emph{generalized} Cahn-Hilliard equation, which is based on the
following objections:
\begin{itemize}
\item Fundamental physical laws should account for the work associated with each operative kinematical process;
\item There is no clear separation of the balance law \eqref{conslaw} and the constitutive equation \eqref{post};
\item Forces that are associated with microscopic configurations of atoms are not considered in the derivation of the classical Cahn-Hilliard equation.
\end{itemize}
According to Gurtin there should exist so called 'microforces' whose
work accompanies changes in the order parameter $\psi$. The
microforce system is characterized by the microstress $\xi\in\R^n$
and scalar quantities $\pi$ and $\ga$ which represent internal and
external microforces, respectively. The main assumption in
\cite{Gur} is that $\xi$, $\pi$ and $\ga$ satisfy the (local)
\emph{microforce balance}
\beq\label{balance}\diver\xi+\pi+\ga=0,\eeq which can be motivated
from a static point of view, see \cite{Gur} for more details. In a
next step we want to derive constitutive equations, which relate the
quantities $j$, the flux of the order parameter, $\xi$ and $\pi$ to
the fields $\psi$ and $\mu$. The technique used in \cite{Gur} for
this derivation is based on the balance equation \eqref{balance} and
a (local) dissipation inequality, which is a direct consequence of
the first and the second law of thermodynamics, that is, the energy
balance
$$\difft \int_\Om e\ dx=-\int_{\pa\Om}q\cdot \nu\ d\si+\int_\Om r\
dx+\calW(\Om)+\calM(\Om),$$ and
$$\difft\int_\Om S\ dx\ge -\int_{\pa\Om}\frac{q}{\th}\cdot \nu\
d\si+\int_\Om\frac{r}{\th}\ dx,$$ cf. \cite[Appendix A]{Gur}. The
second law of thermodynamics is also known as the
\emph{Clausius-Duhem inequality}. Here $e$ is the internal energy,
$S$ is the entropy, $\th$ is the absolute temperature, $q$ is the
heat flux, $r$ is the heat supply, $\calW(\Om)$ is the rate of
working on $\Om$ of all forces exterior to $\Om$ and $\calM(\Om)$ is
the rate at which energy is added to $\Om$ by mass transport. Let
$F$ be the free energy density, depending on the vector
$z=(\psi,\nabla\psi,\mu,\nabla\mu,\pa_t\psi)$. Then the second law
of thermodynamics (in its mechanical version as considered by Gurtin
\cite{Gur}) reads
$$\difft \int_\Om F(z)\ dx\le -\int_{\pa\Om}\mu j(z)\cdot\nu\
d\si+\int_{\pa\Om}\xi\cdot\nu \pa_t\psi\ d\si+\int_\Om\mu m\
dx+\int_\Om\ga\pa_t\psi\ dx,$$ with $m$ being the external mass
supply. Making use of Green's formula, we obtain
\begin{multline*}\difft \int_\Om F(z)\ dx\le -\int_{\Om}(\nabla\mu\cdot
j(z)+\mu\diver j)\ dx\\
+\int_{\Om}(\diver \xi\pa_t\psi+\xi\cdot\nabla\pa_t\psi)\ dx
+\int_\Om\mu m\ dx+\int_\Om\ga\pa_t\psi\ dx.\end{multline*} in
presence of external mass supply $m$, \eqref{conslaw} will be
modified to \beq\label{modconslaw}\pa_t\psi+\diver j=m.\eeq In view
of \eqref{balance} and \eqref{modconslaw} we obtain the dissipation
inequality $$\difft \int_\Om F(z)\ dx\le
\int_{\Om}(\mu\pa_t\psi-j\cdot\nabla\mu-\pi\pa_t\psi+\xi\cdot\nabla\pa_t\psi)\
dx.$$ This in turn yields the following local dissipation inequality
$$\pa_t F(z)\le
\mu\pa_t\psi-j\cdot\nabla\mu-\pi\pa_t\psi+\xi\cdot\nabla\pa_t\psi,$$
for all fields $\psi$ and $\mu$, this means, we have
\beq\label{dissineq2} (\pa_\psi
F+\pi-\mu)\dot{\psi}+(\pa_{\nabla\psi}F-\xi)\cdot\nabla\dot{\psi}+\pa_\mu
F\dot{\mu}+\pa_{\nabla\mu}F\nabla\dot{\mu}+\pa_{\dot{\psi}}F\ddot{\psi}+\nabla\mu\cdot
j\le 0, \eeq where $\dot{u}=\pa_t u$ and $\ddot{u}=\pa_t^2u$ for a
smooth function $u$. This local inequality needs to be satisfied for
all smooth fields $\psi$ and $\mu$. Hence we have necessarily
$$F(z)=F(\psi,\nabla\psi)\quad\text{and}\quad\xi(\psi,\nabla\psi)=\pa_{\nabla\psi}F(\psi,\nabla\psi)$$
and there remains the inequality
$$(\pa_\psi F+\pi-\mu)\dot{\psi}+\nabla\mu\cdot j\le 0$$
whose general solution is given by (cf. \cite[Appendix B]{Gur})
$$\pa_\psi F+\pi-\mu=-\beta\dot{\psi}-c\cdot\nabla\mu\quad\text{and}\quad j=-a\dot\psi-B\nabla\mu,$$
with \emph{constitutive moduli} $\beta(z)$ (scalar), $a(z),c(z)$
(vectors), $B(z)$ (matrix) and the constraint that the matrix
\beq\label{tensor}\begin{bmatrix}\beta & c^\textsf{T}\\a &
B\end{bmatrix}\eeq is positive semidefinite. For convenience we
assume that $\beta$ is constant and $a,c$ and $B$ do only depend on
$x$ instead of $z$, whence we deal with an approximation of the
constitutive moduli $\beta(z),a(z),c(z),B(z)$. In particular, if the
free energy density $F$ is given by
$F(\psi,\nabla\psi)=\frac{1}{2}|\nabla\psi|^2+\Phi(\psi)$ we obtain
the following semilinear \emph{Cahn-Hilliard-Gurtin} equations.
\begin{align}\begin{split}\label{CHG}
\pa_t\psi-\diver (B\nabla\mu)-\diver(a \pa_t\psi)=f,\quad t>0,\ x\in\Om,\\
\mu-c\cdot\nabla\mu+\De \psi-\beta\pa_t\psi-\Phi'(\psi)=g,\quad
t>0,\ x\in\Om,\end{split}\end{align} where $\Om\subset\R^n$ is open
and bounded with boundary $\Ga=\pa\Om\in C^3$. We want to emphasize
that for the special case $B=I$, $a=c=0$ and $\beta=0$, we obtain
the classical Cahn-Hilliard equation or the \emph{viscous}
Cahn-Hilliard equation if $\beta>0$.

Let us point out that we will also deal with a quasilinear version of \eqref{CHG} in Section 5. To be precise, we will consider the system
\begin{align}\begin{split}\label{CHGquasilin}
\pa_t\psi-\diver (b(x,\psi,\nabla\psi)\nabla\mu)-\diver(a(x,\psi,\nabla\psi) \pa_t\psi)=f,\quad t>0,\ x\in\Om,\\
\mu-c(x,\psi,\nabla\psi)\cdot\nabla\mu+\De \psi-\beta\pa_t\psi-\Phi'(\psi)=g,\quad
t>0,\ x\in\Om.\end{split}\end{align}
In this paper, we are interested in solutions of \eqref{CHG} and \eqref{CHGquasilin} subject
to the Neumann boundary conditions for $(\psi,\mu)$, having optimal $L_p$-regularity in the sense
$$\psi\in H_{p}^1(J;H_p^1(\Om))\cap L_{p}(J;H_p^3(\Om)),$$
and
$$\mu\in L_{p}(J;H_p^2(\Om)),$$
for given functions $f\in L_{p}(J;L_p(\Om))$ and $g\in
L_{p}(J;H_p^1(\Om))$, where $J=[0,T]$. We will always use the following assumptions for the semilinear problem \eqref{CHG}.
\begin{itemize}
\item $a,c\in C^1(\overline{\Omega})^n$,
\item $\diver a(x)=\diver c(x)=0$ in $\Omega$,
\item $(a(x)|\nu(x))=(c(x)|\nu(x))=0$ on $\partial\Omega$,
\item $\beta>0$, $B=bI$, with $b\in C^1(\overline{\Omega})$,
\item there is a constant $\ep>0$, such that the estimate
$$\beta
z_0^2+(a+c|z_1)z_0+(Bz_1|z_1)\ge\ep(z_0^2+|z_1|^2)$$ is
valid for all $(z_0,z_1)\in\R\times\R^n$ and all
$x\in\overline{\Om}$.
\end{itemize}
In Section 2, where we consider $\Omega=\mathbb{R}^n$, we allow for general, positive definite matrices $B$. It is also possible to consider those matrices in all other sections but for the sake of convenience we restrict ourself to the case $B=bI$. Actually this allows to draw back the problem in the half space $\mathbb{R}^n_+$ to the whole space $\mathbb{R}^n$ by means of reflection methods.

Results on existence and uniqueness can be found e.g.\ in the papers of
\textsc{Bonfoh \& Miranville} \cite{BonMir}, \textsc{Miranville}
\cite{Mir0a}, \cite{Mir3}, \textsc{Miranville} \& \textsc{Pi\'{e}trus} \cite{MirPiet06}, \textsc{Miranville, Pi\'{e}trus \&
Rakotoson} \cite{MirPieRak} and \textsc{Miranville} \& \textsc{Zelik} \cite{MirZel09}. In any of these papers the authors
use a variational approach and energy estimates to obtain
global well-posedness in an $L_2$-setting, with periodic boundary conditions for a cuboid in $\R^3$. The
qualitative behavior of solutions of the Cahn-Hilliard-Gurtin
equation has been investigated in \cite{BonMir}, \cite{MirPieRak}
and \cite{MirRou}. In \cite{BonMir} and \cite{MirPieRak} the
authors proved the existence of finite dimensional attractors,
whereas \textsc{Miranville \& Rougirel} \cite{MirRou} showed that
each solution converges to a steady state, again with the help of
the Lojasiewicz-Simon inequality. One assumption
of \textsc{Miranville \& Rougirel} \cite{MirRou} is that the norms
$|a|$, $|c|$ and $|B-I|$ are bounded by a possibly small constant. In the present paper we will give an alternative proof for the relative compactness of the orbit $\{\psi(t)\}_{t\ge 0}$ in $H_2^1(\Omega)$ with the help of semigroup theory and a priori estimates (see Proposition \ref{relcompCHG}).

The present paper is structured as follows. In Section 2 we deal
with a corresponding linearized system to \eqref{CHGquasilin} in the full
space $\R^n$ with constant coefficients. Section 3 is devoted to the
analysis of the linearized system with constant coefficients in the
half space $\R_+^n$. Making use of the optimal regularity results of
Sections 2 and 3 we apply the method of localization and some
perturbation results in Section 4 to derive optimal $L_p$-regularity
for the linearized Cahn-Hilliard-Gurtin equations (i.e.\ \eqref{CHG} with $\Phi'=0$) in an arbitrary bounded domain $\Om\subset\R^n$ with boundary $\pa\Om\in C^3$. In
Section 5 we prove the existence and uniqueness of a local-in-time solution of \eqref{CHGquasilin}. For this purpose it is crucial to have the optimal $L_p$-regularity result from Section 4 at our disposal. To the knowledge of the author there are no results on the local well-posedness of \eqref{CHGquasilin} but only for the case where $a$, $c$ and $b$ depend solely on the order parameter $\psi$, cf.\ \textsc{Miranville} \cite{Mir01}.
In Section 6 we investigate the global well-posedness
of the semilinear system \eqref{CHG}. The basic tools are a
priori estimates and the Gagliardo-Nirenberg inequality. Finally, in Section 7, we show that each solution $\psi(t)$ of \eqref{CHG} converges to a steady state in $H_2^1(\Omega)$ as $t\to\infty$. To this end we will use relative compactness results and the Lojasiewicz-Simon inequality.

\section{The Linear Cahn-Hilliard-Gurtin Problem in $\R^n$}\label{CHGsec1}

In this section we will solve the full space problem
\begin{align}\label{GR}
\pa_tu-\diver (a \pa_tu)&=\diver(B\nabla \mu)+f,\quad t>0,\ x\in\R^n,\nn\\
\mu-c\cdot\nabla\mu&=\beta \pa_tu-\De u+g,\quad t>0,\ x\in\R^n,\\
u(0)&=u_0,\quad t=0,\ x\in\R^n\nn,
\end{align}
where $\beta>0$, $a,c\in\R^n$ and $B\in\R^{n\times n}$. Note
that the matrix \eqref{tensor} is positive semidefinite if and only
if
$$\beta z_0^2+(a+c|z_1)z_0+(Bz_1|z_1)\ge 0$$
holds for all $(z_0,z_1)\in\R\times\R^n$ and all
$x\in\overline{\Om}$. Here $(\cdot|\cdot)$ denotes the usual scalar
product in $\C^n$ and the vector fields $a,c$ as well as the matrix
valued function $B$ are assumed to be smooth. In the sequel we will
use a slightly stronger assumption.\vspace{0.25cm}
\begin{itemize}
\item[\textbf{(H)}] There is a constant $\ep>0$, such that
$$\beta
z_0^2+(a+c|z_1)z_0+(Bz_1|z_1)\ge\ep(z_0^2+|z_1|^2)$$ is
valid for all $(z_0,z_1)\in\R\times\R^n$ and all
$x\in\overline{\Om}$.\vspace{0.25cm}
\end{itemize}
The following result is useful for the analysis of \eqref{GR} (see also \cite[Lemma 5.1]{MirRou}).
\begin{pro}\label{HA}
Let (H) hold. Then
$$(\beta B\xi|\xi)-\frac{1}{2}\left((a\otimes c+c\otimes
a)\xi|\xi\right)\ge\ep\beta|\xi|^2,$$ for all $\xi\in\R^n$.
\end{pro}
\bpr Hypothesis (H) reads
$$\beta z_0^2+(d|z_1)z_0+(Bz_1|z_1)\ge\ep(z_0^2+|z_1|^2),$$
where $d:=a+c$. Observe that the left side of this inequality can be
rewritten as
$$\left(\sqrt{\beta}z_0+\frac{1}{2\sqrt{\beta}}(d|z_1)\right)^2+\left(\left(B-\frac{1}{4\beta}(d\otimes d)\right)z_1\Big|z_1\right).$$
For a fixed $z_1\in\R^n$ we choose $z_0\in\R$ in such a way that the
squared bracket is equal to 0. Thus we obtain the estimate
$$(\beta Bz_1|z_1)-\frac{1}{4}((d\otimes d)z_1|z_1)\ge
\ep\beta|z_1|^2,$$ valid for all $z_1\in\R^n$. By the definition of
$d$ it holds that
$$d\otimes d=a\otimes c+c\otimes a+a\otimes a+c\otimes c,$$
hence we obtain the identity \begin{align*}\beta
B-\frac{1}{2}(a\otimes c+c\otimes a)&=\beta B-\frac{1}{4}(d\otimes
d)+\frac{1}{4}(a\otimes a+c\otimes c-a\otimes c-c\otimes a)\\
&=\beta B-\frac{1}{4}(d\otimes
d)+\frac{1}{4}(a-c)\otimes(a-c).\end{align*} Since the matrix
$(a-c)\otimes(a-c)$ is positive semi-definite we finally obtain the
assertion.

\epr Here is the main result on optimal $L_p$-regularity of
\eqref{GR}.
\begin{thm}\label{thmGR}
Let $1<p<\infty$ and assume that (H) holds true. Then \eqref{GR}
admits a unique solution
$$u\in H_{p}^1(J;H_p^1(\R^n))\cap L_{p}(J;H_p^3(\R^n))=:Z^1,$$
$$\mu\in L_{p}(J;H_p^2(\R^n))=:Z^2,$$
if and only if the data is subject to the following conditions.
\begin{enumerate}
\item $f\in L_{p}(J;L_p(\R^n))=:X^1$,
\item $g\in L_{p}(J;H_p^1(\R^n))=:X^2$,
\item $u_0\in B_{pp}^{3-2/p}(\R^n)=:X_p$.
\end{enumerate}
\end{thm}
\bpr Necessity is clear by substituting the solution $(u,\mu)\in
Z^1\times Z^2$ into the equations $\eqref{GR}_{1,2}$. This yields
the desired regularity for the functions $f,g$. The regularity for
the initial value $u_0$ follows from the trace theorem
$$H_{p}^1(J;H_p^1(\R^n))\cap
L_{p}(J;H_p^3(\R^n))\hookrightarrow C(J;B_{pp}^{3-2/p}(\R^n)),$$
where $B_{pp}^{3-2/p}(\R^n)=(H_p^1(\R^n),H_p^3(\R^n))_{1-1/p,p}$ is
the real interpolation space with exponent $1-1/p$ and parameter
$p$.

To prove sufficiency of the conditions (i)-(iii), we first apply the
operator $(I-\De)^{-1/2}$ to both equations in $\eqref{GR}$ and
define the new functions $w=(I-\De)^{-1/2}u$, $\eta=(I-\De)^{-1/2}
\mu$, $\tilde{f}=(I-\De)^{-1/2}f$, $\tilde{g}=(I-\De)^{-1/2} g$ and
$w_0=(I-\De)^{-1/2} u_0$. Then it holds that
$$\tilde{f}\in L_{p}(J;H_p^1(\R^n)),\quad \tilde{g}\in L_{p}(J;H_p^2(\R^n)),$$
$$w_0\in B_{pp}^{4-2/p}(\R^n)$$
and we are looking for a solution $(w,\eta)$ of the system
\begin{align}\label{GRmod}
w_t-\diver (a w_t)&=\diver(B\nabla \eta)+\tilde{f},\quad t>0,\ x\in\R^n,\nn\\
\eta-c\cdot\nabla\eta&=\beta w_t-\De w+\tilde{g},\quad t>0,\ x\in\R^n,\\
w(0)&=w_0,\quad t=0,\ x\in\R^n\nn,
\end{align}
in the regularity class
$$w\in H_{p}^1(J;H_p^2(\R^n))\cap L_{p}(J;H_p^4(\R^n)),$$
$$\eta\in L_{p}(\R_+;H_p^3(\R^n)).$$
In a next step we want to eliminate the functions $\tilde{g}$ and
$w_0$. To achieve this, let $w^*$ be the unique solution of the
problem
\begin{align*}
\beta w^*_t-\De w^*&=-\tilde{g},\quad t>0,\ x\in\R^n,\\
w^*(0)&=w_0,\quad t=0,\ x\in\R^n,
\end{align*}
with regularity
$$w^*\in H_{p}^1(J;L_p(\R^n))\cap L_{p}(J;H_p^2(\R^n)),$$
if and only if $\tilde{g}\in L_{p}(J\times\R^n)$ and $w_0\in
B_{pp}^{2-2/p}(\R^n)$. Here $J$ denotes the interval $[0,T]$. If we
even have $\tilde{g}\in L_{p}(J;H_p^2(\R^n))$ and $w_0\in
B_{pp}^{4-2/p}(\R^n)$ then by regularity theory we obtain
$$w^*\in H_{p}^1(J;H_p^2(\R^n))\cap L_{p}(J;H_p^4(\R^n)).$$
The pair of functions $(v,\eta)=(w-w^*,\eta)$ should now solve the
problem
\begin{align}\label{GR1}
\pa_tv-\diver (a \pa_tv)&=\diver(B\nabla \eta)+F,\quad t>0,\ x\in\R^n,\nn\\
\eta-c\cdot\nabla\eta&=\beta \pa_tv-\De v,\quad t>0,\ x\in\R^n,\\
v(0)&=0,\quad t=0,\ x\in\R^n,\nn
\end{align}
where $F$ is defined by
$$F=\tilde{f}+w_t^*-\diver(aw_t^*)\in L_{p}(J;H_p^1(\R^n)).$$
In order to solve \eqref{GR1} we take the Laplace transform in the
time variable and the Fourier transform in the spatial variable to
obtain
\begin{align*}
\la(1-i(a|\xi))\hat{v}&=-(B\xi|\xi)\hat{\eta}+\hat{F},\\
(1-i(c|\xi))\hat{\eta}&=(\beta\la+|\xi|^2)\hat{v}.
\end{align*}
This system of algebraic equations can be written in matrix form
$$\underbrace{\begin{bmatrix}\la(1-i(a|\xi)) & (B\xi|\xi)\\-(\beta\la+|\xi|^2) & (1-i(c|\xi))\end{bmatrix}}_{M(\la,\xi)}
\begin{bmatrix}\hat{v}\\ \hat{\eta}\end{bmatrix}=\begin{bmatrix}\hat{F}\\0\end{bmatrix},$$
where $\la\in\Sigma_{\phi},\ \phi>\pi/2$ and $\xi\in\R^n$ such
that $|\la|+|\xi|\neq 0$. Hence the unique solution to these
equations is given by
$$\begin{bmatrix}\hat{v}\\ \hat{\eta}\end{bmatrix}=\frac{1}{m(\la,\xi)}
\begin{bmatrix}(1-i(c|\xi)) & -(B\xi|\xi)\\(\beta\la+|\xi|^2) & \la(1-i(a|\xi))\end{bmatrix}\begin{bmatrix}\hat{F}\\0\end{bmatrix},$$
provided
$$m(\la,\xi):=\det M(\la,\xi)\neq 0.$$
To see this we consider the function
$\tilde{m}(\la,\xi):=m(\la,\xi)/\la$ given by
$$\tilde{m}(\la,\xi)=1-(a|\xi)(c|\xi)+\beta(B\xi|\xi)-i(a+c|\xi)+\beta(B\xi|\xi)|\xi|^2/\la=z_1(\xi)+z_2(\la,\xi),$$
where $z_2(\la,\xi):=\beta(B\xi|\xi)|\xi|^2/\la$. Let $\phi_j=\arg
z_j$; then a short computation shows that
$$|z_1+z_2|\ge C(\phi_1,\phi_2)(|z_1|+|z_2|),$$
provided that $|\phi_1-\phi_2|<\pi$. Here
$$C(\phi_1,\phi_2):=\frac{1}{\sqrt{2}}\min\{1,(1+\cos(\phi_1-\phi_2))^{1/2}\}.$$
From Proposition \ref{HA} and the Cauchy-Schwarz inequality we
obtain
$$\left|\frac{(a+c|\xi)}{1-(a|\xi)(c|\xi)+\beta(B\xi|\xi)}\right|\le C|a+c|\frac{|\xi|}{1+|\xi|^2}\le C|a+c|<\infty,$$
hence $|\phi_1|\le\si<\pi/2$ for all $\xi\in\R^n$. Since
$|\phi_2|=|\arg\la|<\phi$ we have
$$|\phi_1-\phi_2|\le \si+\phi<\pi,$$
provided $\phi>\pi/2$ is sufficiently close to $\pi/2$ and this in
turn yields together with Proposition \ref{HA}
$$|\tilde{m}(\la,\xi)|=|z_1+z_2|\ge C(|z_1|+|z_2|)\ge
C(1+|\xi|^2+|\xi|^4/|\la|)$$ or equivalently
\beq\label{estMP0}|m(\la,\xi)|\ge C(|\la|(1+|\xi|^2)+|\xi|^4).\eeq
Observe that the converse is also true, i.e. there is a constant
$C>0$ such that
$$|m(\la,\xi)|\le C(|\la|(1+|\xi|^2)+|\xi|^4).$$
In particular it holds that $m(\la,\xi)=0$ if and only if
$|\la|+|\xi|=0$.

Next, let $v_0, v_1\in\, _0H_{p}^1(J;H_p^2(\R^n))\cap
L_{p}(J;H_p^4(\R^n))$ be the unique solutions of
\begin{align*}
\pa_t(I-\De)v_0+\De^2v_0&=F-c\cdot \nabla F,\quad t>0,\ x\in\R^n,\\
v_0(0)&=0,
\end{align*}
and
\begin{align*}
\pa_t(I-\De)v_1+\De^2v_1&=(I-\De)^{1/2}F,\quad t>0,\ x\in\R^n,\\
v_1(0)&=0.
\end{align*}
The existence of $v_0$ and $v_1$ may be seen by the
Dore-Venni-Theorem. It follows that
$$\pa_t(I-\De)v+\De^2v=S(\pa_t(I-\De)+\De^2)v_0,$$
and $$(I-\De)^{3/2}\eta=S(I-\De)(\beta\pa_t-\De)v_1,$$ where the
linear operator $S$ is defined by its Fourier-Laplace symbol
$$\hat{S}(\la,\xi)=\frac{\la(1+|\xi|^2)+|\xi|^4}{m(\la,\xi)}.$$
Note that the assertion of Theorem \ref{thmGR} follows if we can
show that $S$ is a bounded operator from $L_{p}(J;L_p(\R^n))$ to
$L_{p}(J;L_p(\R^n))$. This will be a consequence of the classical
Mikhlin multiplier theorem and the Kalton-Weis Theorem \cite[Theorem
4.5]{KaWei}.

It is not difficult to show that the symbol $\hat{S}(\lambda,\xi)$ satisfies the Mikhlin condition \vspace{0.25cm}

\begin{itemize}
\item[\textbf{(M)}]\hspace{2.5cm} $\max_{|\al|\le [n/2]+1}\sup_{\xi\in\R^n}|\xi|^{|\al|}
|\pa_\xi^\al\hat{S}(\la,\xi)|<\infty$,\vspace{0.25cm}
\end{itemize}

\noindent where $\al\in\N_0^n$ is a multiindex and $[s]$ denotes the
largest integer not exceeding $s\in\R$.
The classical Mikhlin multiplier theorem then
implies that $\hat{S}$ is a Fourier multiplier in $L_p(\R^n;\C)$
w.r.t. the variable $\xi$ and this yields a holomorphic uniformly
bounded family
$\{\tilde{S}(\la)\}_{\la\in\Sigma_{\phi}}\subset\calB(L_p(\R^n;\C))$,
$\phi>\pi/2$. By \cite[Theorem 3.2]{GiWei} this family is also
$\calR$-bounded in $L_{p}(J;L_p(\R^n;\C))$ (for the notion of
$\calR$-boundedness we refer the reader to \cite{DHP1}). Finally,
since the operator $\pa_t$ admits a bounded $\calH^\infty$-calculus
with angle $\pi/2$ we obtain from \cite[Theorem 4.5]{KaWei} the
desired property of the operator $S$. For the functions
$u=(I-\De)^{1/2}w$ and $\mu =(I-\De)^{1/2}\eta$, this yields
$$u\in H_{p}^1(J;H_p^1(\R^n))\cap L_{p}(J;H_p^3(\R^n)),$$
as well as
$$\mu\in L_{p}(J;H_p^2(\R^n)),$$
and the proof is complete.

\epr

For later purpose we need a perturbation result. To be precise we
consider coefficients $a,c$ and $B$ with a small deviation from
constant ones, i.e.
$$a(x)=a_0+a_1(x),\ c(x)=c_0+c_1(x),\ B(x)=B_0+B_1(x),$$
with $a_1,c_1\in W_\infty^1(\R^n;\R^n)$, $B_1\in
W_\infty^1(\R^n;\R^{n\times n})$ and
$$|a_1|_\infty+|c_1|_\infty+|B_1|_\infty\le\om.$$
Furthermore we assume that $\diver a_1(x)=\diver c_1(x)=0$ for a.e.
$x\in\R^n$ and that the quadruple $(\beta,a_0,c_0,B_0)$ satisfies
(H). Observe that if $\omega>0$ is sufficiently small, then $(\beta,a(x),c(x),B(x))$ satisfy (H) as well for all $x\in\overline{\Omega}$, with a possibly smaller constant $\ep>0$.

We have the following result.
\begin{cor}\label{corGR}
Under the above assumptions on the coefficients the statement of
Theorem \ref{thmGR} remains true, if $\om>0$ is sufficiently small.
\end{cor}
\bpr By a shift of the function $u$ we may assume that
$u_0=g=0$. For the time being we consider an interval
$J_\de=[0,\de]$, with a suitable small $\de>0$, to be chosen later.
The corresponding function spaces are denoted by $X_\de^j$ and
$Z_\de^j$. Moreover
$$_0Z_\de^1:=\{u\in Z_\de^1:u|_{t=0}=0\}.$$
Assume that we already know a solution $(u,\mu)\in\,_0Z_\de^1\times
Z_\de^2$ of \eqref{GR}. Thanks to Theorem \ref{thmGR} we have a
solution operator $\calS\in\calB(X_\de^1\times X_\de^2\times
X_p;\,_0Z_\de^1\times Z_\de^2)$ for the constant coefficient case
$(\beta,a_0,c_0,B_0)$. With the help of $\calS$ we write the
solution in the following way.
$$\begin{bmatrix}u\\ \mu\end{bmatrix}=\calS\begin{bmatrix}f\\0\\0\end{bmatrix}+\calS \calT\begin{bmatrix}u\\ \mu\end{bmatrix},$$
where
$$\calT\begin{bmatrix}u\\ \mu\end{bmatrix}:=\begin{bmatrix}\diver(a_1(x)\pa_t
u)+\diver(B_1(x)\nabla \mu)\\ c_1(x)\cdot\nabla
\mu\\0\end{bmatrix}.$$ From the boundedness of $\calS$ and since
$\diver a_1(x)=0,\ x\in\R^n$, we obtain the estimate
\begin{align}\begin{split}\label{GR2}
|(u,\mu)|_{Z_\de^1\times Z_\de^2}&\le C(|f|_{X_\de^1}+|\calT(u,\mu)|_{X_\de^2})\\
&\le C(|f|_{X_\de^1}+\om|(u,\mu)|_{Z_\de^1\times
Z_\de^2}+|\nabla\mu|_{L_{p}(J_\de;L_p(\R^n))}),
\end{split}\end{align} for some constant $C>0$. The problem is that
the term $|\nabla\mu|$ does not become small in
$L_{p}(J_\de;L_p(\R^n))$, since the function $\mu$ has no regularity
w.r.t. the variable $t$. However, we have the following result.
\begin{pro}\label{proGR2}
Let $(u,\mu)\in Z_\de^1\times Z_\de^2$ be a solution of \eqref{GR}
with $g=u_0=0$. Assume furthermore that the (variable) coefficients
satisfy the above assumptions. Then there exists a constant $C>0$,
independent of $J_\de$, such that the estimate \beq
|\mu|_{L_{p}(J_\de;H_p^1(\R^n))}\le
C(|f|_{X_\de^1}+|g|_{X_\de^2}+|u|_{L_{p}(J_\de;H_p^2(\R^n))}) \eeq is valid.
\end{pro}
\bpr
The proof follows the lines of the proof of Proposition \ref{proHS2}.
\epr

\noindent Owing to \eqref{GR2} and Proposition \ref{proGR2} we
obtain the estimate
\begin{align}\begin{split}\label{GR3}
|(u,\mu)|_{Z_\de^1\times Z_\de^2}&\le
C(|f|_{X_\de^1}+\om|(u,\mu)|_{Z_\de^1\times
Z_\de^2}+|u|_{L_{p}(J_\de;H_p^2(\R^n))}).
\end{split}\end{align}
The mixed derivative theorem and Sobolev embedding yield
$$_0H_p^1(J_\de;H_p^1(\R^n))\cap
L_p(J_\de;H_p^3(\R^n))\hookrightarrow\,
_0H_p^{1/2}(J_\de;H_p^2(\R^n))\hookrightarrow
L_{2p}(J_\de;H_p^2(\R^n)),$$ hence by H\"older's inequality we
obtain $|u|_{L_p(J_\de;H_p^2(\R^n))}\le\de^{1/2p}C|u|_{Z^1_\de}$ and
the constant $C>0$ does not depend on $\de>0$, since $u|_{t=0}=0$.
Choosing first $\om>0$, then $\de>0$ small enough and shifting back
the function $u$, we obtain from \eqref{GR3} the estimate
$$|(u,\mu)|_{Z_\de^1\times Z_\de^2}\le
C(|f|_{X_\de^1}+|g|_{X_\de^2}+|u_0|_{X_p}),$$ where $C>0$ is some
constant. The latter estimates show that the operator $L\in\calB(Z_\de^1\times
Z_\de^2;X_\de^1\times X_\de^2\times X_p)$, defined by
$$L\begin{bmatrix}u\\ \mu\end{bmatrix}=\begin{bmatrix}\pa_t
u-\diver(a\pa_t u)-\diver (B\nabla\mu)\\
\mu-c\cdot\nabla\mu-\beta\pa_t u+\De u\\u|_{t=0}\end{bmatrix},$$
is injective and has closed range, hence $L$ is a semi-Fredholm operator. Replacing the coefficients
$(\beta,a,c,B)$ by
$$(\beta,a_\tau,c_\tau,B_\tau):=(1-\tau)(\beta,a_0,c_0,B_0)+\t(\beta,a,c,B),\ \tau\in[0,1]$$
we may conclude from the considerations above that for each
$\tau\in[0,1]$ the corresponding operator $L_\tau$ is semi-Fredholm as well. The continuity of the Fredholm index yields that
the index of $L_1=L$ is 0, since $L_0$ is an isomorphism, by Theorem
\ref{thmGR}. A successive application of the above procedure yields
the claim for the time interval $J=[0,T]$. The proof is complete.

\epr


\section{The Linear Cahn-Hilliard-Gurtin Problem in
$\R_+^n$}\label{CHGsec2}

In order to treat the case of a half space, we consider first
constant coefficients which are subject to the following
assumptions: $B=bI$ and $(a|e_n)=(c|e_n)=0$, where $e_n:=[0,\ldots,0,-1]^{\textsf{T}}$ is the outer unit normal at $\partial\R_+^n$. Furthermore we assume that $(\beta,a,c,B)$ satisfy (H), whence it holds that $b\ge \ep>0$. Moreover the
boundary conditions on $a$ and $c$ yield that the last components of
$a$ and $c$ are identically zero. We are interested to solve the
following system in $\R_+^n$.
    \begin{align}\begin{split}\label{HS}
    \pa_tu-\diver (a \pa_tu)&=b\De\mu+f,\quad t>0,\ (x',y)\in\R_+^n,\\
    \mu-c\cdot\nabla\mu&=\beta \pa_tu-\De u+g,\quad t>0,\ (x',y)\in\R_+^n,\\
    \pa_y \mu&=h_1,\quad t>0,\ x'\in\R^{n-1},\ y=0,\\
    \pa_y u&=h_2,\quad t>0,\ x'\in\R^{n-1},\ y=0,\\
    u(0)&=u_0,\quad t=0,\ (x',y)\in\R_+^n.
    \end{split}\end{align}
Note that the conormal boundary condition $(b\nabla\mu|e_n)=h_1$
is equivalent to $-b\partial_y\mu=h_1$, where $b>0$ is constant.
Hence it suffices to consider the boundary condition $\pa_y\mu=h_1$
with some scaled function $h_1$. Concerning optimal $L_p$-regularity
of \eqref{HS} we have the following result.
\begin{thm}\label{HSthm}
Let $1<p<\infty$, $p\neq 3/2$ and assume that (H) holds true. Then
\eqref{HS} admits a unique solution
$$u\in H_{p}^1(J;H_p^1(\R_+^n))\cap L_{p}(J;H_p^3(\R_+^n))=:Z^1,$$
$$\mu\in L_{p}(J;H_p^2(\R_+^n))=:Z^2,$$
if and only if the data is subject to the following conditions.
\begin{enumerate}
\item $f\in L_{p}(J;L_p(\R_+^n))=:X^1$,
\item $g\in L_{p}(J;H_p^1(\R_+^n))=:X^2$,
\item $h_1\in L_{p}(J;W_p^{1-1/p}(\R^{n-1}))=:Y^1$,
\item $h_2\in W_{p}^{1-1/2p}(J;L_p(\R^{n-1}))\cap L_{p}(J;W_p^{2-1/p}(\R^{n-1}))=:Y^2$,
\item $u_0\in B_{pp}^{3-2/p}(\R_+^n)=:X_p$.
\item $\pa_y u_0=h_2|_{t=0}$ if $p>3/2$.
\end{enumerate}
\end{thm}
\bpr The necessity part follows from the equations and trace theory,
cf. \cite{DHP2}. Concerning sufficiency, we first reduce \eqref{HS} to the case $h_1=h_2=u_0=0$. For this purpose we
solve the elliptic problem
\begin{align}\label{HS1}\begin{split}(I-\De_{x'})\eta-\pa_y^2\eta=0,&\quad x'\in\R^{n-1},\ y>0,\\
\pa_y\eta=h_1,&\quad x'\in\R^{n-1},\ y=0.\end{split}\end{align}
Define $\tilde{L}:=(I-\De_{x'})^{1/2}$ in $L_p(\R^{n-1})$, with
$D(\tilde{L})=H_p^1(\R^{n-1})$ and let $L$ denote the natural
extension of $\tilde{L}$ to $L_{p,loc}(\R_+;L_p(\R^{n-1}))$, that is
$D(L)=L_{p,loc}(\R_+;H_p^1(\R^{n-1}))$ and $Lu=\tilde{L} u$ for each
$u\in D(L)$. Then the unique solution $\eta$ of \eqref{HS1} is given
by
$$\eta(y)=-L^{-1}e^{-Ly}h_1.$$
Since $h_1\in
L_{p}(J;W_p^{1-1/p}(\R^{n-1}))=D_L(1-1/p,p)$, we have $e^{-Ly}h_1\in
D(L)$ and therefore $\eta\in L_{p}(J;H_p^2(\R^{n}_+))$, with
$\pa_y\eta|_{y=0}=h_1$. In order to remove $h_2 $ and $u_0$, we
solve the initial boundary value problem
\begin{align}\label{HS2}\begin{split}\beta\pa_tv-\De_{x'} v-\pa_y^2v=0,&\quad t>0,\ x'\in\R^{n-1},\ y>0,\\
\pa_y v=h_2,&\quad t>0,\ x'\in\R^{n-1},\ y=0,\\
v(0)=u_0,&\quad t=0,\ x'\in \R^{n-1},\ y>0.\end{split}\end{align} To
this end we extend $u_0\in B_{pp}^{3-2/p}(\R_+^n)$ to a function
$\tilde{u}_0\in B_{pp}^{3-2/p}(\R^n)$ and solve the heat equation
$$\beta\pa_t \tilde{v}-\De \tilde{v}=0,\ t>0,\ x\in\R^n,\quad
\tilde{v}(0)=\tilde{u}_0,\ t=0,\ x\in\R^n,$$ in
$L_{p}(J;H_p^1(\R^n))$. This yields a solution
$$\tilde{v}\in H_{p}^1(J;H_p^1(\R^n))\cap L_{p}(J;H_p^3(\R^n)).$$
If $v_1:=P\tilde{v}$ denotes the restriction of $\tilde{v}$ to the
half space $\R_+^n$, the function $v_2:=v-v_1$ should solve the
initial boundary value problem
\begin{align}\label{HS3}\begin{split}\beta\pa_tv_2-\De_{x'} v_2-\pa_y^2v_2=0,&\quad t>0,\ x'\in\R^{n-1},\ y>0,\\
\pa_y v_2=\bar{h}_2,&\quad t>0,\ x'\in\R^{n-1},\ y=0,\\
v_2(0)=0,&\quad t=0,\ x'\in \R^{n-1},\ y>0,\end{split}\end{align}
where $\bar{h}_2:=h_2-\pa_y v_1|_{y=0}$. Set
$v_3=(I-\De_{x'})^{1/2}v_2$. Then $v_3$ is a solution of
\begin{align}\label{HS3a}\begin{split}\beta\pa_tv_3-\De_{x'} v_3-\pa_y^2v_3=0,&\quad t>0,\ x'\in\R^{n-1},\ y>0,\\
\pa_y v_3=h_3,&\quad t>0,\ x'\in\R^{n-1},\ y=0,\\
v_3(0)=0,&\quad t=0,\ x'\in \R^{n-1},\ y>0.\end{split}\end{align}
with $h_3=(I-\De_{x'})^{1/2}\bar{h}_2\in\,
_0W_p^{1/2-1/2p}(J;L_p(\R^{n-1}))\cap L_p(J;W_p^{1-1/p}(\R^{n-1}))$.
We define $L=(\beta\pa_t-\De_{x'})^{1/2}$ with natural domain
$$D(L)=\, _0H_{p}^{1/2}(J;L_p(\R^{n-1}))\cap
L_{p}(J;H_p^1(\R_+^{n-1})).$$ Then, the unique solution $v_3$ of
\eqref{HS3a} is given by
$$v_3(y)=-L^{-1}e^{-Ly}h_3,$$
and $h_3\in D_L(1-1/p,p)$. This yields
$$v_3\in\, _0H_{p}^1(J;L_p(\R_+^n))\cap L_{p}(J;H_p^2(\R^n_+)).$$
On the other hand, if we consider the function $v_4:=\pa_y v_2$ as
the solution of
\begin{align}\label{HS3b}\begin{split}\beta\pa_tv_4-\De_{x'} v_4-\pa_y^2v_4=0,&\quad t>0,\ x'\in\R^{n-1},\ y>0,\\
v_4=\bar{h}_2,&\quad t>0,\ x'\in\R^{n-1},\ y=0,\\
v_4(0)=0,&\quad t=0,\ x'\in \R^{n-1},\ y>0,\end{split}\end{align} we
obtain $v_4(y)=e^{-Ly}\bar{h}_2$ and $\bar{h}_2\in D_L(2-1/p,p)$.
This yields
$$v_4\in\, _0H_{p}^1(J;L_p(\R_+^n))\cap L_{p}(J;H_p^2(\R^n_+)).$$
From the regularity of $v_3$ and $v_4$ we may conclude that
$$v_2\in\, _0H_{p}^1(J;H_p^1(\R_+^n))\cap L_{p}(J;H_p^3(\R_+^n)).$$
Now the functions $u_1:=u-v$ and $\mu_1:=\mu-\eta$, with
$v=v_1+v_2$, should solve the system
\begin{align}\label{HS4}
\pa_tu_1-\diver (a \pa_tu_1)&=b\De\mu_1+f_1,\quad t>0,\ x'\in\R^{n-1},\ y>0,\nn\\
\mu_1-c\cdot\nabla\mu_1&=\beta \pa_tu_1-\De u_1+g_1,\quad t>0,\ x'\in\R^{n-1},\ y>0,\nn\\
\pa_y\mu_1&=0,\quad t>0,\ x'\in \R^{n-1},\ y=0,\\
\pa_y u_1&=0,\quad t>0,\ x'\in \R^{n-1},\ y=0,\nn\\
u_1(0)&=0,\quad t=0,\ x'\in\R^{n-1},\ y>0\nn,
\end{align}
with some modified data $f_1\in X^1$ and $g_1\in X^2$. In a next
step we extend the functions $f_1$ and $g_1$ w.r.t. the spatial
variable to $\R^n$ by even reflection, i.e. we set
$$f_2(t,x',y)=\begin{cases}
  f_1(t,x',y),  & \text{if}\ y\ge 0\\
  f_1(t,x',-y), & \text{if}\ y\le 0
\end{cases}\quad\text{and}\quad g_2(t,x',y)=\begin{cases}
  g_1(t,x',y),  & \text{if}\ y\ge 0\\
  g_1(t,x',-y), & \text{if}\ y\le 0
\end{cases}.$$
Thanks to Theorem \ref{thmGR} we can solve the full space problem
\begin{align}\label{HS5}
\pa_tu_2-\diver (a \pa_tu_2)&=b\De\mu_2+f_2,\quad t>0,\ x\in\R^n,\nn\\
\mu_2-c\cdot\nabla\mu_2&=\beta \pa_tu_2-\De u_2+g_2,\quad t>0,\ x\in\R^n,\\
u_2(0)&=0,\quad t=0,\ x\in\R^n\nn,
\end{align}
since $f_2\in L_{p}(J;L_p(\R^n))$ and $g_2\in L_{p}(J;H_p^1(\R^n))$.
This yields a unique solution
$$u_2\in H_{p}^1(J;H_p^1(\R^{n}))\cap
L_{p}(J;H_p^3(\R^n))\quad\text{and}\quad \mu_2\in
L_{p}(J;H_p^2(\R^n)),$$ by Theorem \ref{thmGR}. At this point we
emphasize that the equations $\eqref{HS4}_{1,2}$ are invariant
w.r.t. even reflection on the hyper surface $\R^{n-1}\times\{0\}$ in
the normal variable $y$, due to the structure of the coefficients.
This in turn implies that the solution $(u_2,\mu_2)$ is symmetric,
w.r.t the variable $y$ and this yields necessarily, $\pa_y
u_2|_{y=0}=\pa_y\mu_2|_{y=0}=0$. Denoting by $P$ the restriction of
the solution $(u_2,\mu_2)$ to the half space $\R_+^n$, it follows
that $(u_1,\mu_1)=P(u_2,\mu_2)$ is the unique solution of
\eqref{HS4} and therefore $u=v+u_1$ and $\mu=\eta+\mu_1$ is the
unique solution of \eqref{HS}. The proof is complete.

\epr

For later purposes we will need the following perturbation result.
Let $B_0=b_0I$, $$a(x)=a_0+a_1(x),\ c(x)=c_0+c_1(x),\ B(x)=B_0+B_1(x),\
D(x)=I+D_1(x)$$ with $a_1,c_1\in W_\infty^1(\R_+^n;\R^n)$, $B_1\in
W_\infty^1(\R_+^n;\R^{n\times n})$, $D_1\in
W_\infty^2(\R_+^n;\R^{n\times n})$ and
$$|a_1|_\infty+|c_1|_\infty+|B_1|_\infty+|D_1|_\infty\le\om,$$
for some $\om>0$. Let furthermore $\diver a_1(x)=\diver c_1(x)=0$
for a.e. $x\in\R_+^n$ and
$$(a_0|\nu(x))=(a_1(x)|\nu(x))=(c_0|\nu(x))=(c_1(x)|\nu(x))=0.$$
If the constant coefficients $(\beta,a_0,c_0,B_0)$ satisfy Hypothesis (H) we have the
following result.
\begin{cor}\label{corHS}
Let $1<p<\infty$, $p\neq 3/2$, $\beta>0$ and suppose that the data
satisfies the conditions (i)-(v) of Theorem \ref{HSthm} and
$(D\nabla u_0|e_n)=h_2|_{t=0}$ if $p>3/2$. Under the above assumptions on the
coefficients $(a,c,B,D)$, there exists a unique solution
$$u\in H_p^1(J;H_p^1(\R_+^n))\cap L_p(J;H_p^3(\R_+^n)),$$
$$\mu\in L_p(J;H_p^2(\R_+^n)),$$
of the system
\begin{align}\begin{split}\label{HSpert}
\pa_tu-\diver (a \pa_tu)&=\diver(B\nabla \mu)+f,\quad t>0,\ (x',y)\in\R_+^n,\\
\mu-(c|\nabla\mu)&=\beta \pa_tu-\diver(D\nabla u)+g,\quad t>0,\ (x',y)\in\R_+^n,\\
(B\nabla\mu|e_n)&=h_1,\quad t>0,\ x'\in\R^{n-1},\ y=0,\\
(D\nabla u|e_n)&=h_2,\quad t>0,\ x'\in\R^{n-1},\ y=0,\\
u(0)&=u_0,\quad t=0,\ (x',y)\in\R_+^n,
\end{split}\end{align}
provided $\om>0$ is sufficiently small.
\end{cor}
\begin{proof}
First of all, we reduce \eqref{HSpert} to the case $u_0=0$ as
follows. Extend the initial data $u_0\in B_{pp}^{3-2/p}(\R_+^n)$ to
some $\tilde{u}_0\in B_{pp}^{3-2/p}(\R^n)$ and solve the heat
equation
\begin{align*}
\pa_t v-\De v&=0,\ t>0,\ x\in\R^n,\\
v(0)&=\tilde{u}_0,\ x\in\R^n,
\end{align*}
to obtain a unique solution
$$v\in H_p^1(J;H_p^1(\R^n))\cap L_p(J;H_p^3(\R^n))=Z^1,$$
for some interval $J=[0,T]$. If $(u,\mu)\in Z^1\times Z^2$ is a
solution of \eqref{HSpert}, then the shifted function
$(u-v,\mu)\in\,_0Z^1\times Z^2$ solves \eqref{HSpert} with $u_0=0$
and some modified functions $\tilde{f}\in X^1$, $\tilde{g}\in X^2$
and $\tilde{h}_2\in\,_0Y^2$. Observe that $\tilde{f},\tilde{g}$ and
$\tilde{h}_2$ depend only on $f,g,h_2$ and the fixed function $v\in
Z^1$ from above. In the sequel we will not rename the functions
$u,f,g$ and $h_2$.

By the structure of the coefficients and by trace theory we obtain
the estimate
    \begin{multline*}|(u,\mu)|_{Z_\de^1\times Z_\de^2}\\\le
    C(|f|_{X_\de^1}+|g|_{X_\de^2}+|h_1|_{Y_\de^1}+|h_2|_{Y_\de^2}+\om|(u,\mu)|_{Z_\de^1\times
    Z_\de^2}+|u|_{L_p(J_\de;H_p^2(\R_+^n))}+|\nabla\mu|_{L_p(J_\de;L_p(\R_+^n))}),
    \end{multline*}
with a constant $C>0$ which does not depend on $\de>0$ since
$u|_{t=0}=0$. The derivation of this estimate follows the lines of
the proof of Corollary \ref{corGR}. The term
$|u|_{L_p(J_\de;H_p^2(\R_+^n))}$ is of lower order and may be
estimated by
$$|u|_{L_p(J_\de;H_p^2(\R_+^n))}\le\de^{1/2p}C|u|_{Z^1_\de}\le \de^{1/2p}|(u,\mu)|_{Z^1_\de\times Z_\de^2},$$
hence this term may be compensated by the left side of the latter
estimate if $\de>0$ is small enough. If in addition $\om>0$ is
sufficiently small, the same is true for
$\om|(u,\mu)|_{Z_\de^1\times Z_\de^2}$. To estimate the term $|\nabla\mu|$ in $L_p(J_\de;L_p(\R_+^n))$, we use the following
proposition whose proof is given in the Appendix.
\begin{pro}
\label{proHS2} Let $(u,\mu)\in Z_\de^1\times Z_\de^2$ be a solution
of \eqref{HSpert} with $u_0=0$. Then there exists a constant $C>0$,
independent of $J_\de$, such that the estimate
    \beq
    |\mu|_{L_{p}(J_\de;H_p^1(\R_+^n))}\le
    C(|f|_{X_\de^1}+|g|_{X_\de^2}+|h_1|_{Y_\de^1}+|u|_{L_{p}(J_\de;H_p^2(\R_+^n))})
    \eeq
is valid.
\end{pro}
Now the claim follows by applying a similar homotopy argument as in
the proof of Corollary \ref{corGR}.

\end{proof}

\section{Bounded domains, Localization}

Let $\Om\subset\R^n$ be a bounded domain with boundary $\pa\Om\in
C^3$. In this section we solve the system
\begin{align}\begin{split}\label{Dom}
\pa_tu-\diver (a \pa_tu)&=\diver(b\nabla \mu)+f,\quad t>0,\ x\in\Om,\\
\mu-c\cdot\nabla\mu&=\beta \pa_tu-\De u+g,\quad t>0,\ x\in\Om,\\
b\nabla\mu\cdot\nu&=h_1,\quad t>0,\ x\in\pa\Om,\\
\pa_\nu u&=h_2,\quad t>0,\ x\in\pa\Om,\\
u(0)&=u_0,\quad t=0,\ x\in\Om,
\end{split}\end{align}
with coefficients $a,c\in [C^1(\overline{\Om})]^n$ and $b\in C^1(\overline{\Om})$. We furthermore assume that
$\diver a(x)=\diver c(x)=0,\ x\in\Om$,
$(a(x)|\nu(x))=(c(x)|\nu(x))=0,\ x\in\pa\Om$ and $(\beta,a,c,b)$ satisfy (H). Before we start with
the localization procedure we prove two lemmata, which are
interesting for their own.
\begin{lem}\label{heateqreg}
Let $1<p<\infty$, $p\neq 3/2$, $J=[0,T]$ and $\Om\subset\R^n$ be a
bounded domain with $\partial\Om\in C^3$. Then for each $\beta>0$
the initial-boundary value problem
\begin{align}\label{heateqregprb}\begin{split}
\beta\pa_t u-\De u=f,&\quad t\in J,\ x\in\Om,\\
\pa_\nu u=g,&\quad t\in J,\ x\in\pa\Om,\\
u(0)=u_0,&\quad t=0,\ x\in\Om,
\end{split}\end{align}
admits a unique solution
$$u\in H_p^1(J;H_p^1(\Om))\cap L_p(J;H_p^3(\Om)),$$
if and only if the data are subject to the following conditions.
\begin{enumerate}
\item $f\in L_p(J;H_p^1(\Om))$,
\item $g\in W_p^{1-1/2p}(J;L_p(\pa\Om))\cap
L_p(J;W_p^{2-1/p}(\pa\Om))$,
\item $u_0\in B_{pp}^{3-2/p}(\Om)$,
\item $\pa_\nu u_0=g|_{t=0}$, provided $p>3/2$.
\end{enumerate}
\end{lem}
\bpr The 'only if' part follows from the equations and well known
result in trace theory. Indeed, given a solution
$$u\in H_p^1(J;H_p^1(\Om))\cap L_p(J;H_p^3(\Om)),$$
of \eqref{heateqregprb} it follows directly that $f\in
L_p(J;H_p^1(\Om))$. Furthermore it holds that
$$H_p^1(J;H_p^1(\Om))\cap L_p(J;H_p^3(\Om))\hookrightarrow C(J;(H_p^1(\Om);H_p^3(\Om))_{1-1/p,p})=C(J;B_{pp}^{3-2/p}(\Om)),$$
by trace- and interpolation theory. Hence $u(0)\in
B_{pp}^{3-2/p}(\Om)$. Finally observe that
$$\nabla u\in H_p^1(J;L_p(\Om))\cap L_p(J;H_p^2(\Om)).$$
Taking the trace of $\nabla u$ on $\pa\Om$ yields
$$\nabla u|_{\pa\Om}\in W_p^{1-1/2p}(J;L_p(\pa\Om))\cap
L_p(J;W_p^{2-1/p}(\pa\Om)),$$ the required regularity for $g$.
Finally, since
$$W_p^{1-1/2p}(J;L_p(\pa\Om))\cap
L_p(J;W_p^{2-1/p}(\pa\Om))\hookrightarrow
C(J;B_{pp}^{2-3/p}(\pa\Om)),$$ it follows that $\pa_\nu
u(0)=g|_{t=0}$ in case $p>3/2$. To prove sufficiency of the
conditions (i)-(iv), note that by the results of Sections
\ref{CHGsec1} \& \ref{CHGsec2} the unique solution of the
corresponding full space and half space problem to
\eqref{heateqregprb} possess the desired regularity. Then the claim
for a bounded domain $\Om\subset\R^n$ with $\pa\Om\in C^3$ follows
from localization, change of coordinates and perturbation theory,
cf. \cite{DHP1}.

\epr

The second lemma provides maximal regularity of \eqref{Dom} in case
$a=c=0$ and $b=1$, the so-called \emph{viscous Cahn-Hilliard
equation} in its linear form.
\begin{lem}\label{CHGspecial}
Let $1<p<\infty$, $p\neq 3/2$, $J=[0,T]$ and $\Om\subset\R^n$ be a
bounded domain with $\partial\Om\in C^3$. Then for each $\beta>0$
the system
\begin{align}\label{CHGspecial2}\begin{split}
\pa_tu-\De\mu=f,&\quad t\in J,\ x\in\Om,\\
\mu-\beta\pa_t u+\De u=g,&\quad t\in J,\ x\in\Om,\\
\pa_\nu \mu=h_1,&\quad t\in J,\ x\in\pa\Om,\\
\pa_\nu u=h_2,&\quad t\in J,\ x\in\pa\Om,\\
u(0)=u_0,&\quad t=0,\ x\in\Om,
\end{split}\end{align}
admits a unique solution
$$u\in H_p^1(J;H_p^1(\Om))\cap L_p(J;H_p^3(\Om)),$$
$$\mu\in L_p(J;H_p^2(\Om)),$$
if and only if the data are subject to the following conditions.
\begin{enumerate}
\item $f\in L_p(J;L_p(\Om))$,
\item $g\in L_p(J;H_p^1(\Om))$,
\item $h_1\in L_p(J;W_p^{1-1/p}(\pa\Om))$,
\item $h_2\in W_p^{1-1/2p}(J;L_p(\pa\Om))\cap
L_p(J;W_p^{2-1/p}(\pa\Om))$,
\item $u_0\in B_{pp}^{3-2/p}(\Om)$,
\item $\pa_\nu u_0=h_2|_{t=0}$, provided $p>3/2$.
\end{enumerate}
\end{lem}
\bpr By Lemma \ref{heateqreg} there exists a unique solution
$$v\in H_p^1(J;H_p^1(\Om))\cap L_p(J;H_p^3(\Om)).$$ of the problem
\begin{align*}
\beta\pa_t v-\De v=-g,&\quad t\in J,\ x\in\Om,\\
\pa_\nu v=h_2,&\quad t\in J,\ x\in\pa\Om,\\
v(0)=u_0,&\quad t=0,\ x\in\Om.
\end{align*}
Hence, w.l.o.g. we may assume $g=h_2=u_0=0$ in \eqref{heateqregprb},
with $f$ being replaced by some modified function $\tilde{f}\in
L_p(J;L_p(\Om))$, which depends at most on the fixed functions $f$
and $v$.

Now we want to reduce \eqref{CHGspecial2} to a single equation for
$u$. Suppose that we already know a solution of \eqref{CHGspecial2}.
Inserting $\eqref{CHGspecial2}_1$ into $\eqref{CHGspecial2}_2$
yields the elliptic problem
\begin{align*} \mu-\beta\De\mu&=\beta f-\De u,\quad t\in J,\
x\in\Om,\\
\pa_\nu\mu&=h_1,\quad t\in J,\ x\in\partial\Om,
\end{align*}
for the function $\mu$. It is well-known that for each $\beta>0$ the
latter problem admits a unique solution $\mu\in L_p(J;H_p^2(\Om))$,
provided $(\beta f-\De u)\in L_p(J;L_p(\Om))$ and $h_1\in
L_p(J;W_p^{1-1/p}(\pa\Om))$. Denoting by $\calS$ the corresponding
solution operator, we may write
$$\mu=\calS\begin{bmatrix}\beta f\\h_1\end{bmatrix}-\calS\begin{bmatrix}\De
u\\0\end{bmatrix}.$$ Inserting this expression into
$\eqref{CHGspecial2}_2$ we obtain the problem
\begin{align}\begin{split}\label{CHGspecial3}
\beta\pa_t u-\De u&=h-\tilde{\calS}u,\quad
t\in J,\ x\in\Om,\\
\pa_\nu u&=0,\quad t\in J,\ x\in\pa\Om,\\
u(0)&=0,\quad t=0,\ x\in\Om,\end{split}\end{align} where
$h:=\calS(\beta f,h_1)$ and $\tilde{\calS}u:=\calS(\De u,0)$. Since
$\calS$ is a bounded linear operator from $L_p(J;L_p(\Om))\times
L_p(J;W_p^{1-1/p}(\pa\Om))$ to $L_p(J;H_p^2(\Om))$ it follows that
$\tilde{\calS}$ is bounded and linear from $L_p(J;H_p^2(\Om))$ to
$L_p(J;H_p^2(\Om))$. Thanks to Lemma \ref{heateqreg} there exists a
solution operator $\calT$ of \eqref{heateqregprb} which is a linear
and bounded mapping from
$$L_p(J;H_p^1(\Om))\times\hspace{0.05cm}_0W_p^{1-1/2p}(J;L_p(\pa\Om))\cap
L_p(J;W_p^{2-1/p}(\pa\Om))\times B_{pp}^{3-2/p}(\Om)$$ to
$_0H_p^1(J;H_p^1(\Om))\cap L_p(J;H_p^3(\Om))$. With the help of
$\calT$ we may write
$$u=\calT\begin{bmatrix}h\\0\\0\end{bmatrix}-\calT\begin{bmatrix}\tilde{\calS}u\\0\\0\end{bmatrix}.$$
We estimate
\begin{multline*}
|\calT(\tilde{\calS}u,0,0)|_{Z}\le
C|\tilde{\calS}u|_{L_p(J;H_p^1(\Om))}\le C|u|_{L_p(J;H_p^2(\Om))}\\
\le
CT^{1/2p}|u|_{L_{2p}(J;H_p^2(\Om))}\le CT^{1/2p}|u|_Z,
\end{multline*}
by H\"older's inequality. Here the constant $C>0$ does not depend on
$T>0$, since the time traces at $t=0$ are zero. A Neumann series
argument yields a unique solution $u\in Z$ of \eqref{CHGspecial3} on
a (possibly) small time interval $J=[0,T]$. Since
\eqref{CHGspecial3} is linear and invariant with respect to time
shifts, the solution exists global in time.

\epr

The main result of this section reads as follows.
\begin{thm}\label{lin}
Let $1<p<\infty$, $p\neq 3/2$, $J=[0,T]$. Suppose furthermore that
$a,c\in [C^1(\bar{\Om})]^{n}$, $b\in C^1(\bar{\Om})$. Then
\eqref{Dom} admits a unique solution
$$u\in H_p^1(J;H_p^1(\Om))\cap L_p(J;H_p^3(\Om))=Z^1,\quad \mu\in
L_p(J;H_p^2(\Om))=Z^2,$$ if and only if the data are subject to the
following conditions.
    \begin{enumerate}
    \item $f\in L_p(J;L_p(\Om))=X^1$,
    \item $g\in L_p(J;H_p^1(\Om))=X^2$,
    \item $h_1\in L_p(J;W_p^{1-1/p}(\Ga))=Y^1$,
    \item $h_2\in W_p^{1-1/2p}(J;L_p(\Ga))\cap
    L_p(J;W_p^{2-1/p}(\Ga))=Y^2$,
    \item $u_0\in B_{pp}^{3-2/p}(\Om)=X_p$,
    \item $\pa_\nu u_0=h_2|_{t=0}$ if $p>3/2$.
    \end{enumerate}
\end{thm}
\bpr By Lemma \ref{CHGspecial} we may first reduce \eqref{Dom} to
the case $h_1=h_2=u_0=0$ and some modified functions $f,g$ in the
right regularity classes. We cover $\overline{\Om}$ by finitely many
open sets $U_k$, $k=1,...,N$, which are subject to the following
conditions. \be
\item $U_k\cap\Ga=\emptyset$ and $U_k=B_{r_k}(x_k)$ for all
$k=1,...,N_1$;
\item $U_k\cap\Ga\neq\emptyset$ for $k=N_1+1,...,N.$
\ee We choose next a partition of unity $\{\ph_k\}_{k=1}^N$ such
that $\sum_{k=1}^N\ph_k(x)=1$ on $\overline{\Om}$, $0\le\ph_k(x)\le
1$ and $\text{supp}\ \ph_k\subset U_j$. Note that $(u,\mu)$ is a
solution of \eqref{Dom} if and only if
    \begin{align}\label{Dom2}
    \pa_tu_k-\diver(a\pa_tu_k)&=\diver(b\nabla \mu_k)+f_k+F_k(u,\mu),\quad t\in[0,\de],\ x\in\Om\cap U_k,\ 1\le k\le N,\nn\\
    \mu_k-c\cdot\nabla\mu_k&=\beta \pa_tu_k-\De u_k+g_k+G_k(u,\mu),\quad t\in[0,\de],\ x\in\Om\cap U_k,\ 1\le k\le N\nn\\
    b\nabla\mu_k\cdot\nu&=(b\nabla\ph_k\cdot\nu)\mu,\quad t\in[0,\de],\ x\in\Ga\cap U_k,\ N_1+1\le k\le N\\
    \pa_\nu u_k&=u\pa_\nu\ph_k,\quad t\in[0,\de],\ x\in \Ga\cap U_k,\ N_1+1\le k\le N\nn\\
    u_k(0)&=0,\quad t=0,\ x\in\Om\cap U_k\nn.
    \end{align}
Here we have set $u_k=u\ph_k$, $\mu_k=\mu\ph_k$, $f_k=f\ph_k$,
$g_k=f\ph_k$. The terms $F_k(u,\mu)$ and $G_k(u,\mu)$ are defined by
    $$F_k(u,\mu)=-(a\cdot\nabla\ph_k)\pa_t
u-(\nabla b\cdot\nabla\ph_k)\mu-2b\nabla\ph_k\cdot\nabla\mu-b\mu\De\ph_k,$$
and
    $$G_k(u,\mu)=-(c\cdot \nabla\ph_k)\mu+2\nabla u\nabla\ph_k+u\De\ph_k.$$
In case $k=1,...,N_1$ we have no boundary conditions, i.e. we only
have to consider the first two equations in \eqref{Dom2}. In order
to treat these local problems with the help of Corollary \ref{corGR}
we extend the coefficients from $B_{r_k}(x_k)$ to $\R^n$ in such a
way that $\diver \tilde{a}(x)=\diver \tilde{c}(x)=0,\ x\in\R^n$,
holds for the extended coefficients $\tilde{a}$ and $\tilde{c}$.
Note that w.l.o.g. we may assume $x_k=0$. This follows by a
translation in $\R^n$.

We use the following extension $\tilde{a}$ of $a$
(or $\tilde{c}$ of $c$).
    \beq\label{ext0}
    \tilde{a}^k(x)=
        \begin{cases} a(x),\quad &x\in\overline{B_{r_k}(0)},\\
        a\left(\frac{r_k^2 x}{r^2}\right)-2\left(\xi\Big|a\left(\frac{r_k^2
        x}{r^2}\right)\right)\xi+R(r,\xi)\xi,\quad &x\in \R^n\setminus
        \overline{B_{r_{k}}(0)},
        \end{cases}
    \eeq
where $r=|x|$, $\xi=x/|x|$ and $\xi_j$, $a_j$ denote the components
of $\xi$ and $a$, respectively. The task is to compute the scalar
valued function $R(r,\xi)$. Since $\diver a(x)=0,\ x\in\Om$, the
divergence of $a\left(\frac{r_k^2}{r^2}x\right)$ and
$\left(\xi\Big|a\left(\frac{r_k^2}{r^2}x\right)\right)\xi$ may be
computed to the result
    $$\diver \left[a\left(\frac{r_k^2}{r^2}x\right)\right]=-2\frac{r_k^2}{r^2}\sum_{i,j}\xi_i\xi_j\pa_j a_i\left(\frac{r_k^2
x}{r^2}\right)$$ and
    $$\diver\left[
\left(\xi\Big|a\left(\frac{r_k^2}{r^2}x\right)\right)\xi\right]=\frac{(n-1)}{r}\left(\xi\Big|a\left(\frac{r_k^2
x}{r^2}\right)\right)-\frac{r_k^2}{r^2}\sum_{i,j=1}^n\xi_i\xi_j\pa_ia_j\left(\frac{r_k^2
x}{r^2}\right).$$ The divergence of the last term $R(r,\xi)\xi$ is
given by
    $$\diver \left[R(r,\xi)\xi\right]=\pa_r
R(r,\xi)+\frac{n-1}{r}R(r,\xi).$$ Finally, this yields that $\diver
\tilde{a}^k(x)=0$ if and only if the function $R=R(r,\xi)$ solves
the ordinary differential equation
    $$\pa_r R(r,\xi)+\frac{(n-1)}{r} R(r,\xi)=2\frac{(n-1)}{r}\left(\xi\Big|a\left(\frac{r_k^2}{r}\xi\right)\right),\quad r\ge r_k.$$
In order to achieve $\tilde{a}^k,\tilde{c}^k\in W_\infty
^1(\R^n;\R^n)$, we require $\tilde{a}^k(r_k\xi)=a(r_k\xi)$. This
yields the initial condition $R(r_k,\xi)=2(a(r_k\xi)|\xi)$, hence
the function $R=R(r,\xi)$ is explicitly given by
    $$R(r,\xi)=\frac{r_k^{n-1}}{r^{n-1}}R(r_k,\xi)+\frac{2(n-1)}{r^{n-1}}\int_{r_k}^r
s^{n-2}\left(a\left(\frac{r_k^2}{s}\xi\right)\Big|\xi\right)\
ds,\quad r\ge r_k.$$ Since
    $$2\frac{r_k^{n-1}}{r^{n-1}}(a(r_k\xi)|\xi)=2(a(r_k\xi)|\xi)-2\frac{(n-1)}{r^{n-1}}\int_{r_k}^rs^{n-2}(a(r_k\xi)|\xi)\
ds,$$ we may write
    \begin{multline*}\tilde{a}^k(x)=a\left(\frac{r_k^2 x}{r^2}\right)-2\left(\xi\Big|a\left(\frac{r_k^2
    x}{r^2}\right)-a(r_k\xi)\right)\xi\\
    +\frac{2(n-1)}{r^{n-1}}\int_{r_k}^r
    s^{n-2}\left(a\left(\frac{r_k^2}{s}\xi\right)-a(r_k\xi)\Big|\xi\right)\xi\ ds,
    \end{multline*}
in case $|x|>r_k$. Owing to this identity and the assumption
$a,c\in C^1(\overline{\Om})$, it is evident that there holds
    $$|\tilde{a}^k(x)-a(0)|+|\tilde{c}^k(x)-c(0)|\le\om,$$
for all $x\in\R^n$, where $\om>0$ can be made as small as we wish,
by decreasing the radius $r_k$ of the charts $U_k$,
$k\in\{1,...,N_1\}$.

For the coefficient function $b$ we use the reflection method from
\cite{DHP1}, i.e. we set
    \beq\label{ext1}
    \tilde{b}^k(x)=
        \begin{cases} b(x),\quad &x\in\overline{B_{r_k}(0)},\\
        b\left(r_k^2\frac{x}{|x|^2}\right),\quad &x\in \R^n\setminus
        \overline{B_{r_{k}}(0)}.
        \end{cases}
    \eeq
It may be readily checked that $\tilde{b}^k\in
W_\infty^1(\R^n)$ and that
    $$|b(0)-\tilde{b}^k(x)|\le \om,\quad x\in\R^n,$$ with the same $\om>0$ as above. Hence for each chart $U_k,\
k\in\{1,...,N_1\}$ we have coefficients which fit into the setting
of Corollary \ref{corGR}. Therefore we obtain corresponding solution
operators $S_k^F$ of \eqref{Dom2} such that
    \beq\label{LocSolFS}\begin{bmatrix}u_k\\\mu_k\end{bmatrix}=S_k^F
    \begin{bmatrix}
    f_k+F_k(u,\mu)\\
    g_k+G_k(u,\mu)
    \end{bmatrix},
    \eeq
for each $k\in\{1,\ldots,N_1\}$.

For the remaining charts $U_k$, $k\in\{N_1+1,\ldots,N\}$ we obtain
problems in perturbed half spaces with inhomogeneous Neumann
boundary conditions. For the further analysis we have to understand
how to treat \eqref{Dom} in such a setting. To this end we fix a
point $x_0\in\pa\Om$ and a chart $U(x_0)\cap\pa\Om\neq\emptyset$.
After a composition of a translation and a rotation in $\R^n$, we
may assume that $x_0=0$ and $\nu(x_0)=[0,\ldots,0,-1]=e_n$. Consider
a graph $\rho\in C^3(\R^{n-1})$, having compact support, such that
    $$\{(x',x_n)\in\overline{U(x_0)}\subset\R^n:x_n=\rho(x')\}=\pa\Om\cap
\overline{U(x_0)}.$$ Note that by decreasing the size of the charts
we may assume that $|\nabla_{x'}\rho|_\infty$ is as small as we
like, since $\nabla_{x'}\rho(0)=0$.

For the time being, we only know that $\diver a(x)=\diver c(x)=0$
for all $x\in U(x_0)\cap \Om$. So we have to extend the coefficients
$a$ and $c$ in a suitable way. To this end we first transform the
crooked boundary $U(x_0)\cap\pa\Om$ to a straight line in
$\R^{n-1}\times\{0\}$. This will be done with the help of a suitable
transformation. Let $u(x',x_n)=v(g(x))=v(x',x_n-\rho(x'))$ and
$\mu(x)=\eta(g(x))=\eta(x',x_n-\rho(x'))$ and
$B_{r_0}(x_0)=g(U(x_0))$. Then the differential operators
$a\cdot\nabla u$ and $c\cdot\nabla \mu$ transform as follows.
    $$a(x)\cdot\nabla u(x)=a(x)\cdot (Dg(x)^\textsf{T}\nabla
v(g(x)))=(Dg(x) a(x))\cdot\nabla v(g(x))=\bar{a}(g(x))\cdot\nabla
v(g(x)),$$ and
    $$c(x)\cdot\nabla \mu(x)=c(x)\cdot (Dg^\textsf{T}(x)\nabla
\eta(g(x)))=(Dg(x) c(x))\cdot\nabla
\eta(g(x))=\bar{c}(g(x))\cdot\nabla \eta(g(x)),$$ with
$\bar{a}(x):=Dg(x) a(g^{-1}(x))$ and $\bar{c}(x)=Dg(x)
c(g^{-1}(x))$. The transformed Laplace operator reads
    $$\De u=\diver(DgDg^{\textsf{T}}\nabla v).$$
Similarly we obtain
    $$\diver(b\nabla\mu)=\diver(\bar{B}\nabla\eta),$$
where $\bar{B}(x):=b(g^{-1}(x))Dg(x)Dg^\textsf{T}(x)$, $x\in
g(U(x_0))\cap\overline{\R_+^n}$. Here the matrix $Dg$ is given by
    $$Dg(x)=\begin{bmatrix}I_{n-1} & 0\\-\nabla_{x'}\rho(x')^{\textsf{T}} &
1\end{bmatrix},\ x'\in \R^{n-1},$$ where
$I_{n-1}$ is the identity matrix in $\R^{(n-1)\times(n-1)}$. Observe
that the normal $\nu$ at $U(x_0)\cap\pa \Om$ is given by
    $$\nu(x',\rho(x'))=\frac{1}{\sqrt{1+|\nabla_{x'}\rho|^2}}\begin{bmatrix}\nabla_{x'}\rho\\-1\end{bmatrix}.$$
Therefore it holds that
$\sqrt{1+|\nabla_{x'}\rho(x')|^2}(Dg^\textsf{T})^{-1}\nu=[0,\ldots,0,-1]^\textsf{T}=e_n$,
hence the transformed boundary conditions are
$\bar{B}\nabla\eta\cdot
e_n=\sqrt{1+|\nabla\rho(x')|^2}\Theta^{-1}h_1$ and
    $$DgDg^{\textsf{T}}\nabla v\cdot e_n=\sqrt{1+|\nabla\rho(x')|^2}\Theta^{-1}h_2.$$
Here $\Theta^{-1}$ is defined by $(\Theta^{-1}u)(x):=u(g^{-1}(x)),\
x\in\R_+^n$.

By construction, the transformed coefficients satisfy $\diver
\bar{a}(x)=\diver \bar{c}(x)=0$ for all $x\in
B_{r_0}(x_0)\cap\R_+^n$ and $(\bar{a}(x)|e_n)=(\bar{c}(x)|e_n)=0$
for all $x\in B_{r_0}(x_0)\cap\pa\R_+^{n}$. Now we are in a position
to use the extension \eqref{ext0} in order to extend $\bar{a}$ and
$\bar{c}$ to the whole of $\overline{\R^n_+}$, such that the
divergence condition $\diver\tilde{a}(x)=\diver\tilde{c}(x)=0$ is
preserved for $x\in\R^n_+$. It is furthermore clear by the structure
of \eqref{ext0} that $(\tilde{a}(x)|e_n)=(\tilde{c}(x)|e_n)=0$ holds
for all $x\in\pa\R_+^n=\R^{n-1}\times\{0\}$. The coefficient matrix
$\bar{B}$ can be extended to a matrix $\tilde{B}$ on
$\overline{\R_+^n}$ by the reflection method \eqref{ext1}. In
particular it holds that $\tilde{B}(x_0)=\bar{B}(x_0)=B(x_0)=b(x_0)I$, by construction.

Therefore we have to solve the following perturbed problem in the half space
$\R_+^n$.
    \begin{align}\label{CrooHStransf}
    \pa_tv-\diver (\tilde{a} \pa_tv)&=\diver(\tilde{B}\nabla \eta)+\Theta^{-1}f,\quad t\in[0,\de],\ x\in \R^{n}_+,\nn\\
    \eta-\tilde{c}\cdot\nabla\eta&=\beta \pa_tv-\diver(D\nabla v)+\Theta^{-1}g,\quad t\in[0,\de],\ x\in\R^{n}_+,\nn\\
    (\tilde{B}\nabla\eta|e_n)&=\sqrt{1+|\nabla_{x'}\rho(x')|^2}\Theta^{-1}h_1,\quad t\in[0,\de],\ x'\in\R^{n-1},\ y=0,\\
    (D\nabla v|e_n)&=\sqrt{1+|\nabla_{x'}\rho(x')|^2}\Theta^{-1}h_2,\quad t\in[0,\de],\ x'\in \R^{n-1},\ y=0,\nn\\
    v(0)&=0,\quad t=0,\ x\in\R^{n}_+\nn,
    \end{align}
with $D:=DgDg^{\textsf{T}}\in W_\infty^2(\R^{n-1})$ and some functions $(f,g,h_1,h_2)\in X^1\times X^2\times Y^1\times Y^2$ such that $h_2|_{t=0}=0$. From the extension method above it follows that
    $$|\tilde{a}(x)-a(x_0)|+|\tilde{c}(x)-c(x_0)|+
|\tilde{B}(x)-B(x_0)|\le \om,$$ for all $x\in \overline{\R_+^n}$
where we can choose $\om>0$ arbitrarily small, by decreasing the
radius $r_0>0$ of the ball $B_{r_0}(x_0)=g(U(x_0))$. Furthermore it
holds that $|D(x)-I|\le\om,\ x\in \overline{\R_+^n}$, since we may
choose $|\nabla\rho|_\infty$ as small as we wish. An application of
Corollary \ref{corHS} yields a unique solution operator $S^H$ of
\eqref{CrooHStransf}, hence $\Theta S^H$ is the corresponding
solution operator for the chart $U(x_0)$. At this point we want to
remark that the function $\sqrt{1+|\nabla_{x'}\rho|^2}$ is a
multiplier for the spaces $W_p^{1-1/p}(\R^{n-1})$ and
$W_p^{2-1/p}(\R^{n-1})$, since $\rho\in C^3(\R^{n-1})$ has compact
support.

This above computation yields solution operators $\Theta_kS_k^H$ for
the charts $U_k,\ k\in\{N_1+1,\ldots,N\}$, hence we may write
\beq\label{LocSolHS}\begin{bmatrix}u_k\\\mu_k\end{bmatrix}=\Theta_kS_k^H
\begin{bmatrix}\Theta_k^{-1}(f_k+F_k(u,\mu))\\\Theta_k^{-1}(g_k+G_k(u,\mu))\\\Theta_k^{-1}(B\nabla\ph_k\cdot\nu)\mu\\
\Theta_k^{-1}(u\pa_\nu\ph_k)\end{bmatrix},\eeq for each
$k\in\{N_1+1,\ldots,N\}$. Summing \eqref{LocSolFS} and
\eqref{LocSolHS} over all charts $U_k$, $k\in\{1,\ldots,N\}$, we
obtain
\beq\label{LocSolAll}\begin{bmatrix}u\\\mu\end{bmatrix}=\sum_{k=1}^{N_1}
S_k^F
\begin{bmatrix}f_k+F_k(u,\mu)\\g_k+G_k(u,\mu)\end{bmatrix}+
\sum_{k=N_1+1}^N \Theta S_k^H
\begin{bmatrix}\Theta^{-1}(f_k+F_k(u,\mu))\\\Theta^{-1}(g_k+G_k(u,\mu))\\\Theta^{-1}(B\nabla\ph_k\cdot\nu)\mu\\
\Theta^{-1}(u\pa_\nu\ph_k)\end{bmatrix},\eeq since
$\{\ph_k\}_{k=1}^N$ is a partition of unity. By the boundedness of
the solution operators we obtain the estimate
\beq\label{LocSolAll2}|(u,\mu)|_{Z_\de^1\times Z_\de^2}\le M
(|f|_{X_\de^1}+|g|_{X_\de^2}+|u|_{L_p(J_0;H_p^2(\Om))}+|\pa_tu|_{L_p(J_0;L_p(\Om))}+|\mu|_{L_p(J_0;H_p^1(\Om))}),
\eeq for some constant $M>0$ which is independent of the interval
$J_0=[0,\de]$ under consideration. The term
$|u|_{L_p(J_0;H_p^2(\Om))}$ may be estimated by
$\de^{1/2p}C|u|_{Z_\de^1}$ with some constant $C>0$ being
independent of $J_0$. To estimate the remaining terms we need the
following result.
\begin{pro}\label{locprop3}
There exists a constant $M>0$, independent of $J_0$, such that
$$|\mu|_{L_p(J_0;H_p^1(\Om))}+|\pa_t u|_{L_p(J_0;L_p(\Om))}\le
M(|f|_{X_\de^1}+|g|_{X_\de^2}+|h_1|_{Y_\de^1}+|u|_{L_p(J_0;H_p^2(\Om))}).$$
\end{pro}
\bpr
The proof follows the lines of the proof of Proposition \ref{proHS2}.
\epr

Choosing $\de>0$ sufficiently small, we obtain from \eqref{LocSolAll2} and Proposition \ref{locprop3} the estimate
$$|(u,\mu)|_{Z_\de^1\times Z_\de^2}\le M
(|f|_{X_\de^1}+|g|_{X_\de^2}+|h_1|_{Y_\de^1}+|h_2|_{Y_\de^2}+|u_0|_{X_p}),
$$
for a solution of \eqref{Dom}. This shows that the bounded operator
$L:Z_\de^1\times Z_\de^2\to X_\de^1\times X_\de^2\times \calY_\de$
defined by
    $$L(u,\mu)=\begin{bmatrix}
    \pa_t u-\diver(a\pa_t u)-\diver(B\nabla\mu)\\
    \mu-(c\cdot\nabla\mu)-\beta\pa_t u+\De u\\
    (B\nabla\mu\cdot\nu)\\ \pa_\nu u\\u|_{t=0}
    \end{bmatrix},$$
is injective and has closed range, i.e. it is semi Fredholm. Here
$\calY_\de$ is defined by
    $$\calY_\de:=\{(h_1,h_2,u_0)\in Y_\de^1\times Y_\de^2\times X_p:
    \pa_\nu u_0=h_2|_{t=0},\ p>3/2\},$$
which is a closed linear subspace of the Banach space $Y_\de^1\times Y_\de^2\times X_p$. To
show surjectivity, we apply again the Fredholm argument to the set
of data
    $$(\beta,a_\t,c_\t,B_\t)=(1-\t)(\beta,0,0,I_n)+\t(\beta,a,c,B),\quad \t\in[0,1].$$
The corresponding operators $L_\t$ are semi Fredholm by the above
procedure and by Lemma \ref{CHGspecial} the operator $L_0$ is
bijective. The continuity of the Fredholm index thus yields that the
index of $L_1=L$ is 0 and therefore the operator $L$ is bijective as
well. A successive application of the above arguments yields
existence of a unique solution $(u,\mu)$ of \eqref{Dom} on an
arbitrary bounded interval $[0,T]$. This completes the proof of
Theorem \ref{lin}.

\epr

\section{Local Well-Posedness}

Let $p>n+2$, $f\in X^1$, $g\in X^2$, $h_j\in Y^j,\ j=1,2$ and
$\psi_0\in X_p$ be given such that the compatibility condition
$\pa_\nu\psi_0=h_2|_{t=0}$ is satisfied. In this section we consider
the quasilinear system \beq\label{LWPCHG}\begin{split}
\pa_t\psi-\diver (a(x,\psi,\nabla\psi) \pa_t\psi)&=\diver(b(x,\psi,\nabla\psi)\nabla \mu)+f,\quad t>0,\ x\in\Om,\\
\mu-c(x,\psi,\nabla\psi)\cdot\nabla\mu&=\beta \pa_t\psi-\De \psi+\Phi'(\psi)+g,\quad t>0,\ x\in\Om,\\
b(x,\psi,\nabla\psi)\partial_\nu\mu&=h_1,\quad t>0,\ x\in\Ga,\\
\pa_\nu \psi&=h_2,\quad t>0,\ x\in \Ga,\\
\psi(0)&=\psi_0,\quad t=0,\ x\in\Om, \end{split}\eeq where $\Phi\in
C^{3-}(\R)$. Assume that we have given vector fields
$a,c\in C^1(\bar{\Om};C^{2-}(\R\times\R^n;\R^n))$ and a scalar valued function $b\in C^1(\bar{\Om};C^{2-}(\R\times\R^n;\R))$ such that
    \begin{equation}\label{tilde}
    \tilde{a}(x):=a(x,\psi_0(x),\nabla\psi_0(x)),\ \tilde{b}(x):=b(x,\psi_0(x),\nabla\psi_0(x)),\ \tilde{c}(x):=c(x,\psi_0(x),\nabla\psi_0(x))
    \end{equation}
satisfy the conditions
    \begin{equation}\label{LWPdiv}
    \diver\tilde{a}(x)=\diver\tilde{c}(x)=0,\ x\in\Om,
    \end{equation}
    \begin{equation}\label{LWPbc}
    (\tilde{a}(x)|\nu(x))=(\tilde{c}(x)|\nu(x))=0,\ x\in\pa\Om.
    \end{equation}
Suppose furthermore that $(\beta,\tilde{a},\tilde{c},\tilde{b})$ are subject to Hypothesis (H) for each $x\in\overline{\Om}$. Observe
that for $p>n+2$ we have $\psi_0\in X_p=B_{pp}^{3-2/p}(\Om)\hookrightarrow
C^2(\overline{\Om})$, hence $\tilde{a},\tilde{c}\in
[C^1(\overline{\Om})]^n$ and $\tilde{b}\in
C^1(\overline{\Om})$ and therefore the coefficients, frozen at $\psi_0$, satisfy the assumptions in Theorem \ref{lin}.

Thanks to Theorem \ref{lin} we may define a pair of functions
$(u^*,v^*)\in Z^1\times Z^2$ as the unique solution of the
linearized system \beq\begin{split}
u^*_t-\diver (\tilde{a}u^*_t)&=\diver(\tilde{b}\nabla v^*)+f,\quad t>0,\ x\in\Om,\\
v^*-\tilde{c}\cdot\nabla v^*&=\beta u^*_t-\De u^*+g,\quad t>0,\ x\in\Om,\\
\tilde{b}\nabla v^*\cdot\nu&=h_1,\quad t>0,\ x\in\Ga,\\
\pa_\nu u^*&=h_2,\quad t>0,\ x\in \Ga,\\
u^*(0)&=\psi_0,\quad t=0,\ x\in\Om.
\end{split}\eeq
We set
$$\E_1=Z^1(T)\times
Z^2(T),\quad\hspace{0.05cm}_0\E_1=\{(u,v)\in\E_1:u|_{t=0}=0\},$$
$$\E_0=X^1(T)\times X^2(T)\times Y^1(T)\times
Y^2(T),\quad\hspace{0.05cm}_0\E_0=\{(f,g,h_1,h_2)\in
\E_0:h_2|_{t=0}=0\}$$ and denote by $|\cdot|_1$ and $|\cdot|_0$ the
canonical norms in $\E_1$ and $\E_0$, respectively. We define a
linear operator $\L:\E_1\to\E_0$ by
$$\L(u,v)=\begin{bmatrix} \pa_tu-\diver (\tilde{a} \pa_tu)-\diver(\tilde{b}\nabla v)\\
v-\tilde{c}\cdot\nabla v-\beta \pa_tu+\De u\\
\tilde{b}\nabla v\cdot\nu\\
\pa_\nu u
\end{bmatrix}$$
and a nonlinear function
$G:\hspace{0.05cm}_0\E_1\times\E_1\to\hspace{0.05cm}_0\E_0$ by
$$G((u,v),(u^*,v^*))=
\begin{bmatrix} G_1((u,v),(u^*,v^*))+G_2((u,v),(u^*,v^*))\\
G_3((u,v),(u^*,v^*))+G_4((u,v),(u^*,v^*))\\
G_5((u,v),(u^*,v^*))\\
0,
\end{bmatrix}
$$
where
$$G_1(u,u^*)=\diver[(a(x,u+u^*,\nabla(u+u^*))-\tilde{a})\pa_t(u+u^*)],$$
$$G_2((u,v),(u^*,v^*))=\diver[(b(x,u+u^*,\nabla(u+u^*))-\tilde{b})\nabla(v+v^*)],$$
$$G_3((u,v),(u^*,v^*))=(c(x,u+u^*,\nabla(u+u^*))-\tilde{c})\cdot\nabla(v+v^*),$$
$$G_4(u,u^*)=\Phi'(u+u^*),$$
and
$$G_5((u,v),(u^*,v^*))=[\tilde{b}-b(x,u+u^*,\nabla(u+u^*))]\nabla(v+v^*)\cdot\nu.$$

Considering $\L$ as an operator from $_0\E_1$ to $_0\E_0$, we obtain
from Theorem \ref{lin} that $\L$ is a bounded isomorphism and by the
open mapping theorem $\L$ is invertible with bounded inverse
$\L^{-1}$. It is easily seen that $(\psi,\mu):=(u+u^*,v+v^*)$ is a
solution of \eqref{LWPCHG} if and only if
$$\L(u,v)=G((u,v),(u^*,v^*))\ \text{or equivalently}\
(u,v)=\L^{-1}G((u,v),(u^*,v^*)).$$ Consider a ball $\B_r\subset
\hspace{0.05cm}_0\E_1$ where $r\in (0,1]$ will be fixed later.
Define a nonlinear operator by
$\calT(u,v):=\L^{-1}G((u,v),(u^*,v^*))$. To apply the contraction
mapping principle we have to show that $\calT\B_r\subset\B_r$ and
that there exists a constant $\ka<1$ such that the contractive
inequality
\beq\label{contrineq}|\calT(u,v)-\calT(\bar{u},\bar{v})|_1\le\ka|(u,v)-(\bar{u},\bar{v})|_1\eeq
holds for all $(u,v),(\bar{u},\bar{v})\in\B_r$. The following
proposition is crucial to prove the desired properties of the
operator $\calT$.
\begin{pro}\label{lipschest}
Let $p>n+2$, $J=[0,T]$ and assume $\Phi\in C^{3-}(\R)$. Then there exists a constant $C>0$, independent of $T$
and $r$, and functions $\mu_j=\mu_j(T)$ with $\mu_j(T)\to 0$ as
$T\to 0$, $j=1,\ldots,5$ such that for all
$(\tilde{u}_1,\tilde{v}_1),(\tilde{u}_2,\tilde{v}_2)\in\B_r$ the
following statements hold. \be \setlength{\itemsep}{1ex}
\item
$|G_1(\tilde{u}_1,u^*)-G_1(\tilde{u}_2,u^*)|_{X^1}\le
C(r+\mu_1(T))|(\tilde{u}_1,\tilde{v}_1)-(\tilde{u}_2,\tilde{v}_2)|_1$;
\item $|G_2((\tilde{u}_1,\tilde{v}_1),(u^*,v^*))-G_2((\tilde{u}_2,\tilde{v}_2),(u^*,v^*))|_{X^1}\le
C(r+\mu_2(T))|(\tilde{u}_1,\tilde{v}_1)-(\tilde{u}_2,\tilde{v}_2)|_1$;
\item $|G_3((\tilde{u}_1,\tilde{v}_1),(u^*,v^*))-G_3((\tilde{u}_2,\tilde{v}_2),(u^*,v^*))|_{X^2}\le
C(r+\mu_3(T))|(\tilde{u}_1,\tilde{v}_1)-(\tilde{u}_2,\tilde{v}_2)|_1$;
\item $|G_4(\tilde{u}_1,u^*)-G_4(\tilde{u}_2,u^*)|_{X^2}\le
C\mu_4(T)|(\tilde{u}_1,\tilde{v}_1)-(\tilde{u}_2,\tilde{v}_2)|_1$;
\item
$|G_5((\tilde{u}_1,\tilde{v}_1),(u^*,v^*))-G_5((\tilde{u}_2,\tilde{v}_2),(u^*,v^*))|_{Y^1}\le
C(r+\mu_5(T))|(\tilde{u}_1,\tilde{v}_1)-(\tilde{u}_2,\tilde{v}_2)|_1.$
\ee
\end{pro}

\bpr Define the ball $\B_r(u^*,v^*)\subset\E_1$ by means of
$$\B_r(u^*,v^*):=\{(u,v)\in\E_1:(u,v)=(\tilde{u},\tilde{v})+(u^*,v^*),\
(\tilde{u},\tilde{v})\in\B_r\}.$$ Let $(u_j,v_j)\in B_r(u^*,v^*),\
j\in\{1,2\}$. Observe that
$$|u_j-u^*|_{\infty,X_p}\le C_0|u_j-u^*|_{Z^1}\le r,$$
with some $C_0>0$, which is independent of $T>0$. This yields
$$|u_j|_{\infty,X_p}\le C r+|u^*|_{\infty,X_p}\le
C_0+|u^*|_{\infty,X_p}=:R,$$ since $r\in (0,1]$. To prove the first
part, note that
\begin{align*}
\diver[(a(x,u_1,\nabla u_1)-\tilde{a})\pa_t
u_1]&-\diver[(a(x,u_2,\nabla
u_2)-\tilde{a})\pa_t u_2]\\
&=(a(x,u_1,\nabla u_1)-\tilde{a})\cdot\nabla\pa_tu_1-(a(x,u_2,\nabla
u_2)-\tilde{a})\cdot\nabla\pa_tu_2\\
&+\diver(a(x,u_1,\nabla u_1)-\tilde{a})\pa_t
u_1-\diver(a(x,u_2,\nabla u_2)-\tilde{a})\pa_t u_2.
\end{align*}
Next we have
\begin{multline*}
(a(x,u_1,\nabla u_1)-\tilde{a})\cdot\nabla\pa_tu_1-(a(x,u_2,\nabla
u_2)-\tilde{a})\cdot\nabla\pa_tu_2\\
=(a(x,u_1,\nabla u_1)-a(x,u_2,\nabla u_2))\cdot\nabla\pa_t
u_1+(a(x,u_2,\nabla u_2)-\tilde{a})\cdot(\nabla\pa_t u_1-\nabla\pa_t
u_2).
\end{multline*}
Therefore we may estimate
\begin{multline*}|(a(\cdot,u_1,\nabla
u_1)-a(\cdot,u_2,\nabla u_2))\cdot\nabla\pa_t u_1|_{X^1}\\
\le |a(\cdot,u_1,\nabla u_1)-a(\cdot,u_2,\nabla
u_2)|_{\infty,\infty}(|\nabla\pa_t u_1-\nabla\pa_t
u^*|_{X^1}+|\nabla\pa_t u^*|_{X^1})\\
\le L(R)C(r+\mu(T))(|u_1-u_2|_{\infty,\infty}+|\nabla u_1-\nabla
u_2|_{\infty,\infty})\\
\le L(R)C(r+\mu(T))|u_1-u_2|_{Z^1},
\end{multline*}
as well as
\begin{multline*}|(a(\cdot,u_2,\nabla
u_2)-\tilde{a})\cdot(\nabla\pa_t u_1-\nabla\pa_t u_2)|_{X^1}\\
\le (|a(\cdot,u_1,\nabla u_1)-a(\cdot,u^*,\nabla
u^*)|_{\infty,\infty}+|a(\cdot,u^*,\nabla
u^*)-\tilde{a}|_{\infty,\infty})(|\nabla\pa_t u_1-\nabla\pa_t
u_2|_{X^1}\\
\le L(R)C(r+\mu(T))|u_1-u_2|_{Z^1},
\end{multline*}
where $\mu(T):=\max\{|\nabla\pa_t
u^*|_{X^1},|u^*-\psi_0|_{\infty,X_p}\}\to 0$ as $T\to 0$ since
$u^*\in Z^1$ is fixed and $u^*|_{t=0}=\psi_0$. For the remaining
terms we use the identity \beq\label{divident}\diver(a(x,u,\nabla
u))={\diver}_xa(x,u,\nabla u)+\pa_za(u,\nabla u)\cdot\nabla u+\pa_q
a(u,\nabla u):\nabla^2 u,\eeq where $a=a(x,z,q),\
q=[q_1,\ldots,q_n]^{\textsf{T}}$
$$\pa_q a(u,\nabla u):\nabla^2 u:=\sum_{i,j=1}^n\pa_{q_i}a_j(u,\nabla
u)\pa_i\pa_j u.$$ Furthermore we make use of
\begin{multline*}
\diver(a(x,u_1,\nabla u_1)-\tilde{a})\pa_t u_1-\diver(a(x,u_2,\nabla
u_2)-\tilde{a})\pa_t u_2\\
=(\diver(a(x,u_1,\nabla u_1))-\diver(a(x,u_2,\nabla u_2)))\pa_t
u_1+\diver(a(x,u_2,\nabla u_2)-\tilde{a})(\pa_t u_1-\pa_t u_2).
\end{multline*}
Let us first estimate $\diver(a(x,u_1,\nabla
u_1))-\diver(a(x,u_2,\nabla u_2))$ in $L_\infty(0,T;L_\infty(\Om))$.
By \eqref{divident} we obtain
\begin{align*}
|\diver(&a(\cdot,u_1,\nabla u_1))-\diver(a(\cdot,u_2,\nabla
u_2))|_{\infty,\infty}\\
&\le|\pa_za(\cdot,u_1,\nabla u_1)-\pa_za(\cdot,u_2,\nabla
u_2)|_{\infty,\infty}|\nabla
u_1|_{\infty,\infty}+|\pa_za(\cdot,u_2,\nabla
u_2)|_{\infty,\infty}|\nabla u_1-\nabla
u_2|_{\infty,\infty}\\
&\hspace{0.2cm}+|\pa_qa(\cdot,u_1,\nabla
u_1)-\pa_qa(\cdot,u_2,\nabla u_2)|_{\infty,\infty}|\nabla^2
u_1|_{\infty,\infty}+|\pa_q a(\cdot,u_2,\nabla
u_2)|_{\infty,\infty}|\nabla^2 u_1-\nabla^2
u_2|_{\infty,\infty}\\
&\le C(R,u^*)|u_1-u_2|_1,
\end{align*}
where $C(R,u^*)>0$ and depends only on $R$ and the fixed function
$u^*\in Z^1$ but not on $T$ and $r$; recall that $r\in (0,1]$ and
$(u_1-u_2)|_{t=0}=0$. Furthermore
$$|\pa_t u_1|_{X^1}\le |\pa_t u_1-\pa_t u^*|_{X^1}+|\pa_t u^*|_{X^1}\le C(r+|\pa_t u^*|_{X^1}),$$
with $C>0$ being independent of $T$ and $|\pa_t u^*|_{X^1}\to 0$ as
$T\to 0$. In a similar way we obtain
\begin{align*}
|\diver(&a(\cdot,u_2,\nabla u_2))-\diver(a(\cdot,\psi_0,\nabla
\psi_0))|_{\infty,\infty}\\
&\le|\pa_za(\cdot,u_2,\nabla u_2)-\pa_za(\cdot,\psi_0,\nabla
\psi_0)|_{\infty,\infty}|\nabla
\psi_0|_{\infty}+|\pa_za(\cdot,u_2,\nabla
u_2)|_{\infty,\infty}|\nabla u_2-\nabla
\psi_0|_{\infty,\infty}\\
&\hspace{0.2cm}+|\pa_qa(\cdot,u_2,\nabla
u_2)-\pa_qa(\cdot,\psi_0,\nabla \psi_0)|_{\infty,\infty}|\nabla^2
\psi_0|_{\infty}+|\pa_q a(\cdot,u_2,\nabla
u_2)|_{\infty,\infty}|\nabla^2 u_2-\nabla^2
\psi_0|_{\infty,\infty}\\
&\le C(R,u^*)|u_2-\psi_0|_{\infty,X_p}\\
&\le C(R,u^*)(r+|u^*-\psi_0|_{\infty,X_p}).
\end{align*}
Note that $|u^*-\psi_0|_{\infty,X_p}\to 0$ as $T\to 0$ since
$u^*|_{t=0}=\psi_0$. Finally it holds that $|\pa_t u_1-\pa_t
u_2|_{X^1}\le |u_1-u_2|_1$. This proves (i). Statements (ii) and
(iii) follow in a very similar way, while (v) follows from trace
theory and (ii).
To prove (iv), we use the condition $\Phi\in C^{3-}(\R)$ to conclude
\begin{align*}
|\Phi'&(u_1)-\Phi'(u_2)|_{X^2}\le
T^{1/p}(|\Phi'(u_1)-\Phi'(u_2)|_{\infty,\infty}+|\nabla\Phi'(u_1)-\nabla\Phi'(u_2)|_{\infty,\infty}\\
&\le
T^{1/p}C(R)(|u_1-u_2|_{\infty,\infty}+|\Phi''(u_1)|_{\infty,\infty}|\nabla
u_1-\nabla
u_2|_{\infty,\infty}+|u_2|_{\infty,\infty}|\Phi''(u_1)-\Phi''(u_2)|_{\infty,\infty})\\
&\le T^{1/p}C(R,u^*)(|u_1-u_2|_{\infty,\infty}+|\nabla u_1-\nabla
u_2|_{\infty,\infty})\\
&\le T^{1/p}C(R,u^*)|u_1-u_2|_1,
\end{align*}
where $C(R,u^*)>0$ does not depend on $T>0$ and $r\in (0,1]$. The
proof is complete.

\epr

With the help of Proposition \ref{lipschest} we are able to prove
the desired properties of the operator $\calT$ defined above. We
first care about the contraction mapping property.
\beq\label{contr}\begin{split} |\calT(u_1,v_1)-\calT(u_2,v_2)|_1&\le
|\L^{-1}||G((u_1,v_1),(u^*,v^*))-G((u_2,v_2),(u^*,v^*))|_0\\
&\le
|\L^{-1}|\Big(|G_1(u_1,u^*)-G_1(u_2,u^*)|_{X^1}\\
&\hspace{2cm}+|G_2((u_1,v_1),(u^*,v^*))-G_2((u_2,v_2),(u^*,v^*))|_{X^1}\\
&\hspace{2cm}+|G_3((u_1,v_1),(u^*,v^*))-G_3((u_2,v_2),(u^*,v^*))|_{X^2}\\
&\hspace{2cm}+|G_4(u_1,u^*)-G_4(u_2,u^*)|_{X^2}\\
&\hspace{2cm}+|G_5((u_1,v_1),(u^*,v^*))-G_5((u_2,v_2),(u^*,v^*))|_{Y^1}\Big)\\
&\le C(r+\mu(T))|(u_1,v_1)-(u_2,v_2)|_1.
\end{split}
\eeq where $\mu=\mu(T)$ is a function with the property that
$\mu(T)\to 0$ as $T\to 0$ and $C>0$ is a constant which does not
depend on $T>0$. Thus, if $T>0$ and $r\in(0,1]$ are sufficiently
small we obtain \eqref{contrineq}. The self mapping property can be
shown in a similar way. The above computation yields
\beq\label{selfm}\begin{split}
|\calT(u,v)|_1&\le |\calT(u,v)-\calT(0,0)|_1+|\calT(0,0)|_1\\
&\le C\left((r+\mu(T))|(u,v)|_1+|G((0,0),(u^*,v^*))|_0\right)\\
&\le C\left((r+\mu(T))r+|G((0,0),(u^*,v^*))|_0\right).
\end{split}
\eeq Since $G((0,0),(u^*,v^*))$ is a fixed function in $\E_0$ it
follows that $|G((0,0),(u^*,v^*))|_{0}\to 0$ as $T\to 0$, whence
$\calT\B_r\subset\B_r$, provided that $T>0$ and $r\in(0,1]$ are
small enough. The contraction mapping principle yields a unique
fixed point $(\hat{u},\hat{v})\in\hspace{0.05cm}_0\E_1$ or
equivalently $(\psi,\mu):=(\hat{u}+u^*,\hat{v}+v^*)\in\E_1$ is the
unique local solution of \eqref{LWPCHG}. Therefore we have the
following result.
\begin{thm}\label{LWPthm}
Let $p>n+2$, $J_0=[0,T_0]$ and suppose that $\Phi\in C^{3-}(\mathbb{R})$,
$a,c\in C^1(\bar{\Om};C^{2-}(\R\times\R^n;\R^n))$ and $b\in C^1(\bar{\Om};C^{2-}(\R\times\R^n;\R))$. Then there exists an interval $J=[0,T]\subset J_0$, such that \eqref{LWPCHG} admits a unique solution
$$\psi\in H_p^1(J;H_p^1(\Om))\cap L_p(J;H_p^3(\Om))=Z^1,\quad \mu\in
L_p(J;H_p^2(\Om))=Z^2,$$ if the data are subject to the
following conditions.
    \begin{enumerate}
    \item $f\in L_p(J;L_p(\Om))=X^1$,
    \item $g\in L_p(J;H_p^1(\Om))=X^2$,
    \item $h_1\in L_p(J;W_p^{1-1/p}(\Ga))=Y^1$,
    \item $h_2\in W_p^{1-1/2p}(J;L_p(\Ga))\cap
    L_p(J;W_p^{2-1/p}(\Ga))=Y^2$,
    \item $\psi_0\in B_{pp}^{3-2/p}(\Om)=X_p$,
    \item $\pa_\nu \psi_0=h_2|_{t=0}$,
    \item $(\beta,\tilde{a},\tilde{b},\tilde{c})$ satisfy (H) for all $x\in\bar{\Omega}$ as well as \eqref{LWPdiv} and \eqref{LWPbc}.
    \end{enumerate}
\end{thm}
\begin{rem}\label{remlwp}
An inspection of the proof of Theorem \ref{LWPthm} shows that the assumption $p>n+2$ can be relaxed to $p>(n+2)/3$ in the semilinear case, i.e.\ if $(a,b,c)$ are independent of $\psi$ and $\nabla\psi$. Indeed, it remains to estimate the nonlinearity $\Phi'(\psi)$ in $L_p(0,T;H_p^1(\Omega))$. However, in the sequel we will always assume the stronger condition $p>n+2$.
\end{rem}

\section{Global Well-Posedness}\label{GWPCHGSEC}

Let $n\le 3$ and $p>n+2$ according to Theorem \ref{LWPthm}. In this section we consider the semilinear version of \eqref{LWPCHG}, i.e.\ we assume that $a=a(x)$,
$c=c(x)$ and $B=b(x)I$. Then, a successive application
of Theorem \ref{LWPthm} yields a maximal interval of
existence $J_{\max}=[0,T_{\max})$ for the solution $(\psi,\mu)$ of
\eqref{LWPCHG}, i.e. \eqref{LWPCHG} admits a unique solution and
    $$\psi\in H_p^1(J;H_p^1(\Om))\cap L_p(J;H_p^3(\Om)),\quad \mu\in
    L_p(J;H_p^2(\Om)),$$
for each interval $J=[0,T]\subset J_{\max}$.

Suppose $T_{\max}<\infty$ and let $J=[0,T]\subset [0,T_{\max})$. We start with an a priori estimate for the
solution $\psi\in Z^1$ on the maximal interval of existence
$J_{\max}$. To do so we multiply $\eqref{LWPCHG}_1$ by $\mu$,
$\eqref{LWPCHG}_2$ by $-\pa_t\psi$ and integrate by parts to obtain
\beq\label{aprio1}\int_\Om\Big(\pa_t\psi\mu+(B\nabla\mu|\nabla\mu)+(a|\nabla\mu)\pa_t\psi\Big)\
dx=\int_\Om \mu f\ dx+\int_\Ga\mu h_1\ d\Ga\eeq and
\beq\label{aprio2}\int_\Om\Big(-\pa_t\psi\mu+(c|\nabla\mu)\pa_t\psi+\beta
|\pa_t\psi|^2+\frac{1}{2}\frac{\pa}{\pa
t}|\nabla\psi|^2+\frac{\pa}{\pa t}\Phi(\psi)\Big)\
dx=\int_\Ga\pa_t\psi h_2\ d\Ga-\int_\Om \pa_t\psi g\ dx,\eeq since
$(a|\nu)=0$ on $\pa\Omega$. Adding \eqref{aprio1} and \eqref{aprio2}
yields the equation
\begin{multline}\label{eneq3}
\difft\left(\frac{1}{2}|\nabla\psi|_2^2+\int_\Om\Phi(\psi)\
dx\right)+\beta |\pa_t\psi|_2^2+(a+c|\pa_t\psi\nabla\mu)_2
+(B\nabla\mu|\nabla\mu)_2\\=\int_\Om \mu f\ dx+\int_\Ga\mu h_1\
d\Ga+\int_\Ga\pa_t\psi h_2\ d\Ga-\int_\Om \pa_t\psi g\
dx.\end{multline} From Assumption (H) with $z_0=\pa_t\psi$ and
$z_1=\nabla\mu$ it follows that
$$\beta |\pa_t\psi|_2^2+(a+c|\pa_t\psi\nabla\mu)_2
+(B\nabla\mu|\nabla\mu)_2\ge
\ep(|\pa_t\psi|_{2}^2+|\nabla\mu|_2^2).$$ For the first and the
second integral in \eqref{eneq3} we apply H\"older's inequality as
well as the Poincar\'{e}-Wirtinger inequality to obtain
$$\int_\Om\mu f\ dx\le C|f|_2\left(|\nabla\mu|_2+|\int_\Om\mu\
dx|\right)\quad \text{and}\quad \int_\Ga\mu h_1\ d\Ga\le
C|h_1|_{2,\Ga}\left(|\nabla\mu|_2+|\int_\Om\mu\ dx|\right).$$ The
integral $\int_\Om\mu\ dx$ can be computed in the following way.
Since $\diver c=0$ in $\Om$ and $(c|\nu)=0$ on $\Ga$ we have
$$\int_\Om (c|\nabla\mu)\ dx=\int_\Ga (c|\nu)\mu\
d\Ga-\int_\Om\mu\diver c\ dx=0,$$ hence it follows from
$\eqref{LWPCHG}_1$, $\eqref{LWPCHG}_2$ and the boundary conditions
that
\begin{equation*}\begin{split}\int_\Om\mu\ dx&=\beta\int_\Om\pa_t\psi\
dx+\int_\Om\Phi'(\psi)\ dx+\int_\Om g\
dx\\
&=\int_\Om\Phi'(\psi)\ dx+\int_\Om g\ dx+\beta\left(\int_\Om f\
dx+\int_\Ga h_1\ d\Ga\right).\end{split}\end{equation*} Assume in
addition
    \beq\label{growthPhi}\Phi(s)\ge
    -\frac{\eta}{2}s^2-c_0,\quad
    s\in\R,
    \eeq
where $c_0>0$ and $0<\eta<\la_1$, with $\la_1>0$ being the first
nontrivial eigenvalue of the negative Neumann Laplacian and
    \beq\label{AssPhi}|\Phi'(s)|\le (c_1\Phi(s)+c_2s^2+c_3)^\th,\quad
    \text{for all $s\in\R$},
    \eeq
and some constants $c_i>0,\ \th\in (0,1)$. This yields
$$|\int_\Om\mu\ dx|\le \int_\Om(c_1\Phi(\psi)+c_2|\psi|^2+c_3)^\th\
dx+c(|g|_{1}+|h_1|_{1,\Ga}+|f|_{1}).$$ By the last estimate, Young's
inequality and the Poincar\'{e} inequality it holds that
\begin{equation}\label{GWPCHG1}\begin{split}
\int_\Om \mu f\ dx+\int_\Ga \mu h_1\ d\Ga&\le
C(\de)\left(|\nabla\psi|_2^2+\int_\Om\Phi(\psi)\
dx+|f|_2^q+|h_1|_{2,\Ga}^q+|g|_{2}^2+1\right)+\de|\nabla\mu|_2^2,
\end{split}\end{equation}
where $q:=\max\{2,\frac{1}{1-\th}\}$ and $\de>0$ may be arbitrarily
small. For the term $\int_\Om\pa_t\psi g\ dx$ in \eqref{eneq3} we
apply Young's inequality one more time to obtain
\beq\label{GWPCHG2}\int_\Om\pa_t\psi g\ dx\le
\de|\pa_t\psi|_{2}^2+C(\de)|g|_2^2.\eeq Integrating \eqref{eneq3}
with respect to $t$ and choosing $\de>0$ small enough, we obtain
together with \eqref{GWPCHG1} and \eqref{GWPCHG2} the estimate
\begin{multline}\label{enineq3}
\frac{1}{2}|\nabla\psi(t)|_2^2+\int_\Om\Phi(\psi(t))\ dx+C_1
(|\pa_t\psi|_{2,2}^2+|\nabla\mu|_{2,2}^2)\\ \le
C_2\left(\int_0^t\left(\frac{1}{2}|\nabla\psi(\tau)|_2^2+\Phi(\psi(\tau))\right)\
d\tau+|f|_{q,2}^q+|h_1|_{q,2,\Ga}^q+|g|_{2,2}^2+1\right)\\+\int_0^t\int_\Ga\pa_t\psi
h_2\ d\Ga\ d\tau.\end{multline} In order to treat the last double
integral, we have to assume more regularity for the function $h_2$.
To be precise, we assume that $$h_2\in H_p^1(J;L_p(\Ga))\cap
L_p(J;W_p^{2-1/p}(\Ga))\hookrightarrow C(J;L_p(\Ga)).$$ Due to this
fact, we may integrate the last term in \eqref{enineq3} by parts to
the result \beq\label{GWPCHG2a}\int_0^t\int_\Ga\pa_t\psi h_2\ d\Ga\
d\tau=\int_\Ga\psi(t)h_2(t)\ d\Ga-\int_\Ga\psi_0h_2|_{t=0}\
d\Ga-\int_0^t\int_\Ga\psi\pa_th_2\ d\Ga\ d\tau,\eeq where we also
made use of Fubini's theorem. For the first term we use Young's
inequality, the embedding $H_2^1(\Om)\hookrightarrow L_2(\Ga)$ and
the fact that \beq\label{conspsi}\int_\Om\psi(t)\ dx=\int_\Om\psi_0\
dx+\int_0^t\int_\Om f\ dx\ d\tau+\int_0^t\int_\Ga h_1\ d\Ga\
d\tau.\eeq This yields
\begin{equation*}\begin{split}\int_\Ga\psi(t)h_2(t)\ d\Ga&\le\de|\psi(t)|_{H_2^1(\Om)}^2+C(\de)|h_2(t)|_{2,\Ga}^2\\
&\le \delta
C|\nabla\psi(t)|_{2}^2+C(\delta)\left(|h_2|_{\infty,2,\Ga}^2+|f|_{1,1}+|h_1|_{1,1,\Ga}+|\psi_0|_1\right).\end{split}\end{equation*}
Observe that we have $h_2|_{t=0}=\pa_\nu\psi_0\in
B_{pp}^{2-3/p}(\Ga)\hookrightarrow L_2(\Ga)$ and, by trace
theory,
$$B_{pp}^{3-2/p}(\Om)\hookrightarrow
B_{pp}^{3-3/p}(\Ga)\hookrightarrow L_2(\Ga).$$ It follows that the
integral $\int_\Ga\psi_0 h_2|_{t=0}\ d\Ga$ converges. Finally,
concerning the last term in \eqref{GWPCHG2a} we apply Young's
inequality one more time to the result
\begin{align*}\int_0^t\int_\Ga\psi\pa_th_2\ d\Ga\
d\tau&\le\frac{1}{2}\int_0^t|\psi(\tau)|_{H_2^1(\Om)}^2\
d\tau+\frac{1}{2}|\pa_th_2|_{2,2,\Ga}^2\\
&\le C\int_0^t|\nabla\psi(\tau)|_{2}^2\
d\tau+C(T,f,h_1,\pa_th_2,\psi_0),
\end{align*}
where we used again \eqref{conspsi}. Set
$$E(u)=\frac{1}{2}|\nabla u|_2^2+\int_\Om\Phi(u)\ dx,\quad u\in
H_2^1(\Om).$$ Then by the above estimates there exist some constants
$C_j>0$ such that
$$E(\psi(t))+C_1
(|\pa_t\psi|_{2,2}^2+|\nabla\mu|_{2,2}^2)\le C_2\int_0^t
E(\psi(\tau))\ d\tau+C_3(T,f,g,h_1,h_2,\pa_th_2,\psi_0),$$
for all $t\in [0,T]$, provided that $\de>0$ is sufficiently small. With the help of
\eqref{growthPhi} it follows that $E(u)$ is bounded from below for
all $u\in H_2^1(\Om)$, hence we may apply Gronwall's lemma to the
result that $E(\psi(\cdot))$ is bounded on $J_{\max}=[0,T_{\max})$.
Applying \eqref{growthPhi} one more time and using the fact that
$|\int_\Om\psi(t,x)\ dx|\le C$ it holds that
$$\psi\in L_\infty(J_{\max};H_2^1(\Om)).$$
Note that in the semilinear case the following estimate for the maximal solution $(\psi,\mu)$ of \eqref{LWPCHG} holds
\begin{multline}\label{GWP2}
|\psi|_{Z^1(T)}+|\mu|_{Z^2(T)}\\ \le C\left(|\Phi'(\psi)|_{X^2(T)}+|f|_{X^1(T)}+|g|_{X^2(T)}+|h_1|_{Y^1(T)}
+|h_2|_{Y^2(T)}+|\psi_0|_{X_p}\right).
\end{multline}
Here the constant $C>0$ does not depend on $T\in (0, T_{\max})$. Suppose that $\Phi'(\psi)$ satisfies the estimate
\begin{equation}\label{GWP3}
|\Phi'(\psi)|_{X^2(T)}\le C(T)|\psi|_{Z^1(T)}^\kappa|\psi|_{L_\infty(0,T_{\max};H_2^1(\Omega))}^m,
\end{equation}
for some $\kappa\in (0,1)$ and $m>0$, where $C(T)>0$ and $\sup_{T\in[0,T_{\max})}C(T)<\infty$. Substituting \eqref{GWP3} into \eqref{GWP2} yields
$$|\psi|_{Z^1(T)}\le M\left(1+|\psi|_{Z^1(T)}^\kappa\right),$$
where $M>0$ does not depend on $T\in (0, T_{\max})$.This in turn yields that $|\psi|_{Z^1(T_{\max})}$ is bounded, since $\kappa\in (0,1)$. Therefore $\psi(T_{\max})\in B_{pp}^{3-2/p}(\Omega)$ is well-defined and we may continue the maximal solution $(\psi,\mu)$ beyond the point $T_{\max}$, which is a contradiction to the maximality of $T_{\max}$.

It remains to show the validity of \eqref{GWP3}. We start with the term $\nabla\Phi'(\psi)=\Phi''(\psi)\nabla\psi$ in $L_p(\Omega)^n$. It holds that
$$|\Phi''(\psi)\nabla\psi|_p\le
|\Phi''(\psi)|_{3p/2}|\nabla\psi|_{3p},$$ by H\"older's
inequality. Assume that there exists a constant $C>0$ such that
\begin{equation}
\label{growthPhi2}|\Phi''(s)|\le C(1+|s|^\alpha),
\end{equation}
for all $s\in\R$ and some $\al\ge 1$, where $\al<4$ in case $n=3$.
Then we have
$$|\Phi''(\psi)\nabla\psi|_p\le C(1+|\psi|_{3\alpha
p/2}^\alpha)|\nabla\psi|_{3p}.$$ Applying the Gagliardo-Nirenberg
interpolation inequality we obtain
$$|\psi|_{3\alpha p/2}\le C|\psi|_{H_p^3(\Omega)}^a|\psi|_{q}^{1-a},$$
provided
$$\frac{n}{q}-\frac{2n}{3\alpha
p}=a\left(3-\frac{n}{p}+\frac{n}{q}\right),\ a\in [0,1].$$ On the
other side we obtain
$$|\nabla\psi|_{3p}\le C|\psi|_{H_p^3(\Omega)}^b|\psi|_{q}^{1-b},$$
provided
$$1-\frac{n}{3p}+\frac{n}{q}=
b\left(3-\frac{n}{p}+\frac{n}{q}\right),\ b\in [1/3,1].$$ Chose $q$
in such a way, that $H_2^1(\Om)\hookrightarrow L_q(\Om)$, i.e.
$n/q\ge n/2-1$. Thus $q$ may be arbitrarily large if $n\in\{1,2\}$
and $q\le 6 $ in case $n=3$.
If $n=3$, let
    \begin{equation}\label{GWP0.1}\frac{\alpha n}{2}<q<\min\left\{6,\frac{3\alpha
    p}{2}\right\},
    \end{equation}
while in case $n=1,2$ we require
    \begin{equation}\label{GWP0.2}
    \frac{\alpha n}{2}<q<\frac{3\alpha p}{2}.
    \end{equation}
This is possible, since $n<3p$ for $n\le 3$ and $\alpha
n/2<6$ if $n=3$, since in this case we assume $\alpha<4$. Now it
follows that
\begin{equation}\label{GWP1}|\psi|_{3\alpha p/2}^\alpha|\nabla\psi|_{3p}\le
C|\psi|_{H_p^3(\Om)}^{a\alpha+b}|\psi|_{q}^{1+\alpha(1-a)-b}.\end{equation}
To gain something from this inequality we require $a\alpha+b<1$
which is equivalent to
$$\left(3-\frac{n}{p}+\frac{n}{q}\right)>\alpha\left(\frac{n}{q}-\frac{2n}{3\alpha
p}\right)+1-\frac{n}{3p}+\frac{n}{q}=1-\frac{n}{p}+(1+\alpha)\frac{n}{q}.$$
This in turn yields $\al<2q/n$ which is certainly true by
\eqref{GWP0.1} and \eqref{GWP0.2}. With $\kappa:=a\alpha+b\in (0,1)$ we obtain the estimate
$$|\nabla\Phi'(\psi(t))|_{L_p(\Omega)^n}\le C|\psi(t)|_{H_p^3(\Omega)}^{\kappa}|\psi(t)|_{H_2^1(\Omega)}^{m},$$
valid for a.e.\ $t\in [0,T]\subset[0,T_{\max})$ and some $m>0$. Similarly one obtains
$$|\Phi'(\psi(t))|_{L_p(\Omega)^n}\le C|\psi(t)|_{H_p^3(\Omega)}^\kappa|\psi(t)|_{H_2^1(\Omega)}^{m},$$
for a.e.\ $t\in [0,T]\subset[0,T_{\max})$. Finally this yields
\begin{equation}\label{GWP4}
|\Phi'(\psi(t))|_{H_p^1(\Omega)}\le C|\psi(t)|_{H_p^3(\Omega)}^\kappa|\psi(t)|_{H_2^1(\Omega)}^{m},
\end{equation}
for a.e.\ $t\in [0,T]\subset[0,T_{\max})$. Integration of the $p$-th power of \eqref{GWP4} and H\"{o}lder's inequality imply \eqref{GWP3}.

In conclusion we have the following result.
\begin{thm}\label{GWPthm}
Let $p>n+2$, $n\le 3$, $q=\max\{2,\frac{1}{1-\th}\}$, with $\th$
from \eqref{AssPhi}. Suppose that $a,c\in
C^1(\overline{\Om})^n$ and $b\in
C^1(\overline{\Om})$ satisfy condition (H) as well as $\diver a(x)=\diver c(x)=0$, $x\in \Omega$ and $(a(x)|\nu(x))=(c(x)|\nu(x))=0$, $x\in\partial\Omega$. Assume furthermore that
$\Phi\in C^{3-}(\mathbb{R})$ satisfies \eqref{growthPhi}, \eqref{AssPhi} and
\eqref{growthPhi2}. Then there exists a unique global solution
$(\psi,\mu)$ of \eqref{LWPCHG} on $J_0=[0,T_0]$, with
$$\psi\in H_p^1(J_0;H_p^1(\Om))\cap L_p(J_0;H_p^3(\Om))$$
and
$$\mu\in L_p(J_0;H_p^2(\Om)),$$
provided that the data are subject to the following conditions.
\begin{enumerate}
\item $f\in L_p(J_0;L_p(\Om))\cap L_q(J_0;L_2(\Om))$,
\item $g\in L_p(J_0;H_p^1(\Om))$,
\item $h_1\in L_p(J_0;W_p^{1-1/p}(\Ga))\cap L_q(J_0;L_2(\Ga))$,
\item $h_2\in H_p^1(J_0;L_p(\Ga))\cap
L_p(J_0;W_p^{2-1/p}(\Ga))$,
\item $\psi_0\in B_{pp}^{3-2/p}(\Om)$,
\item $\pa_\nu\psi_0=h_2|_{t=0}$.
\end{enumerate} The solution depends continuously on the given data and if $f=g=h_1=h_2=0$, the map $\psi_0\mapsto \psi(t),\
t\in\R_+$, defines a global semiflow on the natural phase manifold
defined by (v) \& (vi).
\end{thm}
\begin{rem}
The assertion of Theorem \ref{GWPthm} remains true if we assume that $p>(n+2)/3$, which is sufficient for the well-posedness of the semilinear model by Remark \ref{remlwp}.
\end{rem}

\section{Asymptotic Behavior}

In this last section we will give a qualitative analysis of global
solutions of the Cahn-Hilliard-Gurtin system
\beq\label{ASCHG}\begin{split}
\pa_t\psi-\diver (a \pa_t\psi)&=\diver(b\nabla \mu),\quad t>0,\ x\in\Om,\\
\mu-c\cdot\nabla\mu&=\beta \pa_t\psi-\De \psi+\Phi'(\psi),\quad t>0,\ x\in\Om,\\
B\nabla\mu\cdot\nu&=0,\quad t>0,\ x\in\Ga,\\
\pa_\nu \psi&=0,\quad t>0,\ x\in \Ga,\\
\psi(0)&=\psi_0,\quad t=0,\ x\in\Om. \end{split}\eeq To be more
precise we will show that each trajectory converges to a
stationary point, i.e. to a solution of the corresponding
stationary system. The so called \emph{Lojasiewicz-Simon
inequality} will play an important role in the proof of this
assertion. Assume that $a,c\in C^1(\overline{\Omega})^n$ and $b\in C^1(\overline{\Omega})$
with $\diver a(x)=\diver c(x)=0,\ x\in\Omega$ and $(a(x)|\nu(x))=(c(x)|\nu(x))=0$, $x\in\partial\Omega$. Suppose that the data $(\beta,a,c,B)$ satisfy
condition (H) for all $x\in\overline{\Omega}$. Moreover we assume that $\Phi\in C^3(\mathbb{R})$ and that it satisfies the estimate
\begin{equation}\label{growthPhi3}
|\Phi'''(s)|\le C(1+|s|^\gamma),\ \text{for all}\ s\in\mathbb{R},
\end{equation}
and some constant $C>0$. Here $\gamma\ge 1$ is arbitrary if $n\in\{1,2\}$ and $\gamma<3$ if $n=3$. At this point we want to remark that \eqref{growthPhi3} already implies \eqref{growthPhi2}.

Let $\psi_0\in H_2^2(\Omega)$ such that $\partial_\nu\psi_0=0$
and let $(\psi,\mu)$ be the unique global solution of
\eqref{ASCHG}. We recall from Section \ref{GWPCHGSEC} the energy
functional
$$E(u)=\frac{1}{2}|\nabla u|_2^2+\int_\Om\Phi(u)\ dx,$$
defined on the energy space
$$V:=\{u\in H_2^1(\Om):\int_\Om u\ dx=0\}.$$
Note that due to $\eqref{ASCHG}_1$ and the boundary condition
$\eqref{ASCHG}_3$ we obtain $\int_\Om\psi\ dx\equiv\int_\Om\psi_0\
dx$, since $(a(x)|\nu(x))=0$ on $\Ga$. If we perform a shift of $\psi$
by means of $\tilde{\psi}=\psi-c$, where $c:=\int_\Om\psi_0\ dx$,
it follows that $\tilde{\psi}$ is again a solution of
\eqref{ASCHG}, provided that the physical potential $\Phi$ is replaced by
$\tilde{\Phi}(s)=\Phi(s+c)$. Additionally it holds that
$\int_\Om\tilde{\psi}\ dx=0$. It follows from \eqref{eneq3} that $E(\psi(\cdot))$ satisfies the equation
$$\difft E(\psi(t))+\beta |\pa_t\psi(t)|_2^2+(a+c|\pa_t\psi(t)\nabla\mu(t))_2
+(B\nabla\mu(t)|\nabla\mu(t))_2=0,$$ for all $t\in\R_+$. Making
again use of Hypothesis (H) we obtain the inequality
\beq\label{enineqCHG2}\difft
E(\psi(t))+\ep\left(|\pa_t\psi(t)|_2^2+|\nabla\mu(t)|_2^2\right)\le
0,\eeq which holds for all $t\in\R_+$. Integrating with respect to
$t$ and making use of \eqref{growthPhi} as well as of the
Poincar\'{e} inequality we obtain the a priori estimates
$$\psi\in L_\infty(\R_+;H_2^1(\Om))\quad\text{and}\quad \pa_t\psi,|\nabla\mu|\in
L_2(\R_+\times\Om).$$
\begin{pro}\label{relcompCHG} The orbit
$\{\psi(t)\}_{t\in\R_+}$ is relatively compact in $V$.
\end{pro}
\bpr We rewrite equation $\eqref{ASCHG}_2$ as follows
$$\beta\pa_t\psi-\De\psi+\psi=\mu-\overline{\mu}-(c(x)|\nabla\mu)+\overline{\mu}+\psi-\Phi'(\psi),$$
where $\overline{\mu}=\frac{1}{|\Om|}\int_\Om\Phi'(\psi)\ dx$. By
the energy estimates above and the Poincar\'{e}-Wirtinger
inequality it holds that
$$f:=\mu-\overline{\mu}+(c|\nabla\mu)\in L_2(\R_+;L_2(\Om)).$$
Furthermore we have
$$g:=\overline{\mu}+\psi-\Phi'(\psi)\in L_\infty(\R_+;L_q(\Om)),$$
where $q=6/(\ga+2)$ is determined by the growth condition
\eqref{growthPhi3} on $\Phi$. The operator $A:=\De-I$ with
domain
$$D(A)=\{u\in H_p^2(\Om):\pa_\nu u=0\ \text{on}\ \Ga\}$$
generates an exponentially stable, analytic $C_0$-semigroup
$\{T(t)\}_{t\in\R_+}$ in $L_p(\Om)$. Therefore
$$T(\cdot)\ast f\in H_2^1(\R_+;L_2(\Om))\cap
L_2(\R_+;H_2^2(\Om))\hookrightarrow C_0(\R_+;H_2^1(\Om)).$$ For
the function $g$ we apply elementary semigroup theory to obtain
$$T(\cdot)\ast g\in C_b(\R_+;H_q^s(\Om)),$$
for each $s\in (0,2)$. The space $H_q^s(\Om)$ embeds compactly into
$H_2^1(\Om)$, if $s$ is chosen close enough to 2. This completes the
proof of relative compactness, since $T(\cdot)\psi_0\in C_0(\mathbb{R}_+;H_2^1(\Omega))$.

\epr \noindent The following proposition provides some
properties of the $\om$-limit set
$$\om(\psi)=\{\ph\in V:\ \exists\ (t_n)\nearrow\infty,\ s.t.\
\psi(t_n)\to \ph\ in\ V\}.$$
\begin{pro}\label{omlimset}
Suppose that $(\psi,\mu)$ is a global solution of \eqref{ASCHG}
and let $\Phi$ satisfy Hypotheses \eqref{growthPhi} and
\eqref{growthPhi3}. Then the following statements hold.
\begin{enumerate}
\item The mapping $t\mapsto E(\psi(t))$ is nonincreasing and the
limit $\lim_{t\to\infty}E(\psi(t))=:E_\infty\in\R$ exists.
\item The $\om$-limit set $\om(\psi)\subset V$ is nonempty, connected, compact and $E$ is
constant on $\om(\psi)$.
\item Every $\psi_\infty\in\om(\psi)$ is a strong solution (in the sense of $L_2$) of the
stationary problem
\begin{equation}\label{statsysCHG}\begin{split}
-\De\psi_\infty+\Phi'(\psi_\infty)&=\mu_\infty,\quad x\in\Om,\\
\pa_\nu\psi_\infty&=0,\quad x\in\Ga,
\end{split}\end{equation}
where $\mu_\infty=\frac{1}{|\Om|}\int_\Om\Phi'(\psi_\infty)\
dx=const$.
\item Each $\psi_\infty\in\om(\psi)$ is a critical point of $E$,
i.e. $E'(\psi_\infty)=0$ in $V^*$, where $V^*$ is the topological
dual space of $V$.
\end{enumerate}
\end{pro}
\bpr Inequality \eqref{enineqCHG2} implies that
$E(\psi(\cdot))$ is nonincreasing with respect to $t$. Furthermore
by \eqref{growthPhi} it follows that $E(u)$ is bounded from below
for all $u\in V$. This proves (i). Assertion (ii) follows easily
from well-known facts in the theory of dynamical systems.

Let $\psi_\infty\in \om(\psi)$. Then there exists a sequence
$(t_n)\nearrow+\infty$ such that $\psi(t_n)\to\psi_\infty$ in $V$
as $n\to\infty$. Since $\pa_t\psi\in L_2(\R_+\times\Om)$ it
follows that $\psi(t_n+s)\to\psi_\infty$ in $L_2(\Om)$ for all
$s\in [0,1]$ and by relative compactness also in $V$. Integrating \eqref{enineqCHG2}
from $t_n$ to $t_n+1$ we obtain
$$E(\psi(t_n+1))-E(\psi(t_n))+\ep\int_0^1\int_\Om\left(|\nabla\mu(t_n+s,x)|^2+|\pa_t\psi(t_n+s,x)|^2\right)\
dx\ ds\le 0.$$ Letting $t_n\to+\infty$ yields
$$|\nabla\mu(t_n+\cdot,\cdot)|,\pa_t\psi(t_n+\cdot,\cdot)\to 0\quad\text{in $L_2([0,1]\times\Om)$}.$$ This in turn yields a
subsequence $(t_{n_k})$ such that
$|\nabla\mu(t_{n_k}+s)|,\pa_t\psi(t_{n_k}+s)\to 0$ in $L_2(\Om)$ for
a.e. $s\in [0,1]$. We fix such an $s$, say $s^*\in [0,1]$. The
Poincar\'{e}-Wirtinger inequality implies that
\begin{multline*}|\mu(t_{n_k}+s^*)-\mu(t_{n_l}+s^*)|_2\\\le
C_p\left(|\nabla\mu(t_{n_k}+s^*)-\nabla\mu(t_{n_l}+s^*)|_2+\int_\Om|\Phi'(\psi(t_{n_k}+s^*))-\Phi'(\psi(t_{n_l}+s^*))|\
dx\right),\end{multline*} since $\int_\Om\mu\
dx=\int_\Om\Phi'(\psi)\ dx$. Letting $k,l\to\infty$ and making use
of \eqref{growthPhi3} it follows that $\mu(t_{n_k}+s^*)$ is a
Cauchy sequence in $L_2(\Om)$, hence it admits a limit, which we
denote by $\mu_\infty$. Since the gradient is a closed operator in
$L_2(\Om;\R^n)$ it holds that $\mu_\infty\in H_2^1(\Om)$ and
$\nabla\mu_\infty=0$. Thus $\mu_\infty=const.$ and we have the
identity $\mu_\infty=\frac{1}{|\Om|}\int_\Om\Phi'(\psi_\infty)\
dx$. Finally we multiply $\eqref{ASCHG}_2$ by a function $\ph\in
V$ in $L_2(\Om)$ to the result
\begin{multline}\label{ASCHG1}(\mu(t_{n_k}+s^*),\ph)_2+(c\cdot\nabla\mu(t_{n_k}+s^*),\ph)_2\\=
\beta(\pa_t\psi(t_{n_k}+s^*),\ph)_2-(\De\psi(t_{n_k}+s^*),\ph)_2+(\Phi'(\psi(t_{n_k}+s^*)),\ph)_2.\end{multline}
Taking the limit $t_{n_k}\to\infty$ we obtain
$$a(\psi(t_{n_k}+s^*),\ph)\to
(\mu_\infty-\Phi'(\psi_\infty),\ph)_2,$$ where $a:V\times V\to\R$
is defined by $a(u,v)=(\nabla u,\nabla v)_2$ and $(\cdot,\cdot)_2$
denotes the scalar product in $L_2(\Om)$. Since
$\Phi'(\psi_\infty)\in L_{q}(\Om)$ with $q=6/(\gamma+2)$ it follows
that $\psi_\infty\in D(A_q)=\{u\in H_q^2(\Om):\pa_\nu u=0\}$,
where $A_q$ is the part of the operator $A$ in $L_q(\Om)$ which is
induced by the form $a(u,v)$. Observe that $q>6/5$ by assumption,
whence we may apply a bootstrap argument to conclude
$\psi_\infty\in H_2^2(\Om)$ and $\pa_\nu\psi_\infty=0$ on $\Ga$ (recall that $q>1$ may be arbitrarily large in case $n\in\{1,2\}$).
Going back to \eqref{ASCHG1} we obtain for
$(t_{n_k})\nearrow\infty$ the identity
$$(\nabla\psi_\infty,\nabla\ph)_2+(\Phi'(\psi_\infty),\ph)_2=(\mu_\infty,\ph)_2,$$
for all functions $\ph\in V$. This yields (iii) after integration
by parts. To prove (iv) observe that by \cite[Proposition 5.2]{PrWi11} the functional $E$ is twice continuously Fr\'{e}chet differentiable and its first derivative is given by
$$\langle E'(u),h\rangle_{V^*,V}=\int_\Om \nabla u\nabla h\
dx+\int_\Om\Phi'(u) h\ dx,\quad u,h\in V.$$
Integration by parts finally yields assertion (iv).

\epr

At this point we could simply refer to the paper of \textsc{Miranville} \& \textsc{Rougirel} \cite{MirRou} to prove the main Theorem \ref{mainthmasymp} below. However, for the sake of completeness we provide a proof of this result.

The next proposition is the key for the proof of the
convergence of the orbit $\{\psi(t)\}_{t\ge 0}$ towards a stationary state as
$t\to\infty$.
\begin{pro}[Lojasiewicz-Simon inequality]\label{LSCHG}
Let $\ph\in \om(\psi)$ and assume in addition to \eqref{growthPhi} and \eqref{growthPhi3} that $\Phi$ is real
analytic. Then there exist constants $s\in (0,\frac{1}{2}]$,
$C,\de>0$ such that
$$|E(u)-E(\ph)|^{1-s}\le C|E'(u)|_{V^*},$$
whenever $|u-\ph|_V\le\de$.
\end{pro}
\bpr
This is Proposition 6.6 in \cite{CFP}.
\epr

\noindent Now we are in a position to state the main result of
this section.
\begin{thm}\label{mainthmasymp}
Let $\Phi$ satisfy the conditions \eqref{growthPhi} and
\eqref{growthPhi3}. Assume in addition that $\Phi$ is real
analytic. Then the limit
$$\lim_{t\to\infty}\psi(t)=:\psi_\infty$$
exists in $V$ and $\psi_\infty$ is a strong solution of the
stationary problem \eqref{statsysCHG}.
\end{thm}
\bpr Since each element $\ph\in\om(\psi)$ is a critical point of
$E$, Proposition \ref{LSCHG} implies that the Lojasiewicz-Simon
inequality is valid in some neighborhood of $\ph\in\om(\psi)$. By
Proposition \ref{omlimset} (ii) the $\om$-limit set is compact,
hence there exists $N\in\N$ such that
$$\bigcup_{j=1}^N B_{\de_j}(\ph_j)\supset\om(\psi),$$ where
$B_{\de_j}(\ph_j)\subset V$ are open balls with center
$\ph_i\in\om(\psi)$ and radius $\de_i$. Additionally in each ball
the Lojasiewicz-Simon inequality is valid. It follows from
Proposition \ref{omlimset} (i) and (ii) that the energy functional
$E$ is constant on $\om(\psi)$, i.e. $E(\ph)=E_\infty$, for all
$\ph\in\om(\psi)$. Thus there exists an open set
$U\supset\om(\psi)$ and \emph{uniform} constants $s\in
(0,\frac{1}{2}]$ $C,\de>0$ with
$$|E(u)-E_\infty|^{1-s}\le C|E'(u)|_{V^*},$$
for all $u\in U$. A well-known result in the theory of dynamical
systems sates that the $\om$-limit set is an attractor for the
orbit $\{\psi(t)\}_{t\in\R_+}$. To be precise this means
$$\lim_{t\to\infty}\dist(\psi(t),\om(\psi))=0\quad\text{in}\ V.$$
This implies that there exists some time $t^*\ge 0$ such that
$\psi(t)\in U$ for all $t\ge t^*$ and thus the Lojasiewicz-Simon
inequality holds for the solution $\psi(t)$, i.e.
\beq\label{LSCHG2} |E(\psi(t))-E_\infty|^{1-s}\le
C|E'(\psi(t))|_{V^*},\quad t\ge t^*.\eeq Define a function
$H:\R_+\to\R_+$ by $H(t)=(E(\psi(t)-E_\infty)^s$. Then with
\eqref{enineqCHG2} and \eqref{LSCHG2} it holds that
\begin{equation}\label{HCHG}\begin{split}
-\difft H(t)&=(E(\psi(t))-E_\infty)^{s-1}\left(-\difft
E(\psi(t))\right)\\
&\ge
\ep\frac{|\pa_t\psi(t)|_2^2+|\nabla\mu(t)|_2^2}{(E(\psi(t))-E_\infty)^{1-s}}\\
&\ge
C_\ep\frac{|\pa_t\psi(t)|_2^2+|\nabla\mu(t)|_2^2}{|E'(\psi(t))|_{V^*}}
\end{split}\end{equation}
The first Fr\'{e}chet
derivative of $E$ in $V$ reads
$$\langle E'(u),h\rangle_{V^*,V}=\int_\Om\nabla u\nabla h\
dx+\int_\Om\Phi'(u)h\ dx,$$ for all $(u,h)\in V\times V$. Setting
$u=\psi(t)$ and making use of $\eqref{ASCHG}_2$ we obtain with the
help of H\"{o}lder's inequality, Poincar\'{e}'s inequality and
integration by parts
\begin{equation}\label{HCHG2}\begin{split}\langle
E'(\psi(t)),h\rangle_{V^*,V}&=\int_\Om(\mu(t)-\bar{\mu}(t))h\
dx-\int_\Om
c\cdot\nabla\mu(t) h\ dx-\beta\int_\Om\pa_t\psi(t) h\ dx\\
&\le C(|\nabla\mu(t)|_2+|\pa_t\psi(t)|_2)|h|_2,
\end{split}\end{equation}
since $\diver c(x)=0,\ x\in\Om$ and $(c(x)|\nu(x))=0,\
x\in\pa\Om$. Taking the supremum in \eqref{HCHG2} over all
functions $h\in V$ with norm less than 1 it follows that
$$|E'(\psi(t))|_{V^*}\le C(|\nabla\mu(t)|_2+|\pa_t\psi(t)|_2).$$
We insert this estimate into \eqref{HCHG} to obtain
$$-\difft H(t)\ge C_\ep(|\nabla\mu(t)|_2+|\pa_t\psi(t)|_2).$$
Integrating this inequality from $t^*$ to $\infty$ it follows that
$|\pa_t\psi(\cdot)|_2,|\nabla\mu(\cdot)|_2\in L_1(\R_+)$, since
$H(t)>0$. This implies that the limit
$\lim_{t\to\infty}\psi(t)=:\psi_\infty$ exists firstly in
$L_2(\Om)$ but by relative compactness also in $V$. Finally, by
Proposition \ref{omlimset} (iii) the limit $\psi_\infty$ is a
solution of the stationary problem \eqref{statsysCHG}. The proof
is complete.

\epr

\section{Appendix}

\emph{Proof of Proposition \ref{proHS2}.}

We substitute $\eqref{HSpert}_2$ into $\eqref{HSpert}_1$ to obtain
the elliptic problem \beq\label{EllPrb1}
\mu+\calA(x,\partial)\mu=\diver\left(a\diver(D\nabla
u)\right)-\diver(D\nabla u)+\tilde{f},\quad x\in\R_+^n;\quad
\calB(x,\partial)\mu=h_1,\quad x\in\pa\R_+^n,\eeq with
$$\tilde{f}=\beta f+a\cdot\nabla g-g\in L_p(J_\de\times\R_+^n).$$
Here the differential operators $\calA(x,\partial)$ and $\calB(x,\partial)$ are
defined by
$$\calA(x,\partial)\mu:=-(a+c)\cdot\nabla \mu+\diver(a(c\cdot \nabla
\mu))-\diver(\beta B\nabla \mu),\quad x\in\R_+^n,$$ and
$$\calB(x,\partial)\mu:=B\nabla\mu\cdot\nu,\quad x\in\pa\R_+^n,$$
valid for all $\mu\in H_p^2(\R_+^n)$. It will be convenient to
rewrite the operator $\calA(x,\partial)$ as follows.
$\calA(x,\partial)=\calA_0(x,\partial)+\calA_1(x,\partial)$, where
$$\calA_0(x,\partial)\mu:=-\diver\left(\beta B\nabla\mu-\frac{1}{2}(a\otimes c+c\otimes
a)\nabla\mu\right),\ x\in\R_+^n,$$ and
$$\calA_1(x,\partial)\mu:=a\cdot(\nabla c\nabla\mu)-\frac{1}{2}\Mdiver(a\otimes c+c\otimes
a)\cdot\nabla\mu-(a+c)\cdot\nabla\mu,\ x\in\R_+^n.$$ Actually this
splitting shows that problem \eqref{EllPrb1} is indeed elliptic by
Assumption (H) and Proposition \ref{HA}, provided $\om>0$ is
sufficiently small. Will will now proceed in several
steps.\vspace{0.25cm}

\emph{Step 1}. In this first step we want to reduce \eqref{EllPrb1}
to the case of homogeneous boundary conditions $\calB(x,D)\mu=0$.
Consider the elliptic problem with constant coefficients
\beq\label{EllPrb1a}\la\mu-\diver(\tilde{B}_0\nabla\mu)=f,\quad
x\in\R_+^n,\quad (\tilde{B}_0\nabla\mu|\nu)=g,\quad
x\in\partial\R_+^n, \eeq where $\tilde{B}_0:=\beta
B_0-\frac{1}{2}(a_0\otimes c_0+c_0\otimes a_0)$ and $\la\in\R$ is a
parameter. Note that \eqref{EllPrb1a} is an elliptic problem with a
conormal boundary condition and constant coefficients.

Thanks to Proposition \ref{HA} the matrix $\tilde{B_0}$ is positive
definite. By well known results it follows that for each $f\in
L_p(\R_+^n)$ and $g\in W_p^{1-1/p}(\partial\R_+^n)$ problem
\eqref{EllPrb1a} has a unique solution $\mu\in H_p^2(\R_+^n)$,
provided $\la>0$. We remind that the variable coefficients
$$a(x)=a_0+a_1(x),\quad c(x)=c_0+c_1(x),\quad B(x)=B_0+B_1(x),$$
have a small deviation from the constant ones $a_0,c_0,B_0$, i.e.
$$|a_1|_\infty+|c_1|_\infty+|B_1|_{\infty}\le\om,$$
with $\om>0$ being sufficiently small. Furthermore we have $a,c\in
W_\infty^1(\R_+^n;\R^n)$, $B\in W_\infty^1(\R_+^n;\R^{n\times n})$.
Therefore we may apply perturbation theory to conclude that there
exists $\la_0>0$ such that for each $f\in L_p(\R_+^n)$ and $g\in
W_p^{1-1/p}(\partial\R_+^n)$ the elliptic problem
\beq\label{EllPrb1b}\la\mu-\diver(\tilde{B}\nabla\mu)=f,\quad
x\in\R_+^n,\quad \tilde{B}\nabla\mu\cdot\nu=g,\quad
x\in\partial\R_+^n,\eeq has a unique solution $\mu\in
H_p^2(\R_+^n)$, provided $\la\ge\la_0$ and $\om>0$ is sufficiently
small. Here $\tilde{B}:=\beta B-\frac{1}{2}(a\otimes c+c\otimes a)$.
Note that $\calA_0(x,\partial)\mu=-\diver(\tilde{B}\nabla\mu)$ and
$$(\tilde{B}\nabla\mu|\nu)=(B\nabla\mu|\nu)=\calB(x,\partial)\mu,$$
since $(a|\nu)=(c|\nu)=0$ by assumption. Moreover, the operator
$\calA_1(x,\partial)$ defined above contains only terms of lower order with
$L_\infty$-coefficients. Thus, applying perturbation theory one more
time, there exists $\la_1>0$ such that for each $f\in L_p(\R_+^n)$
and $g\in W_p^{1-1/p}(\partial\R_+^n)$ the problem
\beq\label{EllPrb1c}\la\mu+\calA(x,\partial)\mu=f,\quad x\in\R_+^n;\quad
\calB(x,\partial)\mu=g,\quad x\in\pa\R_+^n,\eeq has a unique solution
$\mu\in H_p^2(\R_+^n)$, provided $\la\ge\la_1$ and $\om>0$ is
sufficiently small.\vspace{0.25cm}

\emph{Step 2.} The results of this first step enable us to reduce
\eqref{EllPrb1} to the case of homogeneous boundary conditions. In
this step we show that the $L_p$-realization $A$ of the boundary
value problem $(\calA,\calB)$ with domain
$$D(A)=\{u\in H_p^2(\R_+^n):\calB(x,\partial)u=0\},$$
is dissipative. First, let $p\ge 2$. Integration by parts yields
\begin{align*}
\Real&\int_{\R^n_+}Aw\ \bar{w}|w|^{p-2}\ dx\\
&=-\int_{\R^n_+}|w|^{p-4}\Real\left(\frac{p}{2}(\tilde{B}\nabla
w\cdot\nabla\bar{w})|w|^2+\left(\frac{p}{2}-1\right)(\tilde{B}\nabla
w\cdot\nabla w)\bar{w}^2\right)\ dx
\end{align*}
for all $w\in D(A)$, since $\diver a(x)=\diver c(x)=0$ in $\Omega$ and $(a(x)|\nu(x))=(c(x)|\nu(x))=0$ on $\partial\Omega$. Setting $\nabla w=u+iv$ and $w=b_1+ib_2$ with $u,v\in\mathbb{R}^n$, $b_j\in\mathbb{R}$, we obtain the estimate
\begin{align*}
\Real\Big(\frac{p}{2}(\tilde{B}\nabla
w\cdot\nabla\bar{w})|w|^2+\Big(\frac{p}{2}-1\Big)(&\tilde{B}\nabla
w\cdot\nabla w)\bar{w}^2\Big)\\ &\ge
\Big[((\tilde{B}u|u)+(\tilde{B}v|v))(b_1^2+b_2^2)\Big]\\
&\ge
\ep\beta(|u|^2+|v|^2)(b_1^2+b_2^2)=\ep\beta|\nabla
w|^2|w|^2.
\end{align*}
Here we made use of Proposition \ref{HA}. This shows that $A$ is dissipative in $L_p(\Omega)$ for $p\ge 2$. If $p\in (1,2)$ we replace $|w|$ by $w_\ep:=\sqrt{|w|^2+\ep}$ for $\ep>0$ in the calculations involving $A$ and then pass to the limit as $\ep\searrow0$.

The dissipativity of $A$ allows us to set $\la=1$ in
\eqref{EllPrb1c}. By the same arguments one can show that the
$L_p$-realization $A_0$ of the elliptic boundary value problem
$(\calA_0,\calB)$ with domain $D(A_0)=D(A)$ is dissipative, too. Indeed, $A_0\mu=\diver (\tilde{B}\nabla\mu)$ and $\tilde{B}$ defined above is a positive definite and symmetric matrix by Proposition \ref{HA}.
Therefore we may also set $\la=1$ in
\eqref{EllPrb1b}.\vspace{0.25cm}

\emph{Step 3.} By the results of Steps 1 \& 2 we may decompose the
unique solution $\mu\in H_p^2(\R_+^n)$ of \eqref{EllPrb1} into
$\mu=\mu_1+\mu_2$, where $\mu_1,\mu_2\in H_p^2(\R_+^n)$ are the
unique solutions of the elliptic problems
\begin{equation}\label{EllPrb1d}
\mu_1+\calA_0(x,\partial)\mu_1=\diver\left((a\diver(D\nabla
u)\right)),\quad x\in\R_+^n;\quad \calB(x,\partial)\mu_1=0,\quad
x\in\pa\R_+^n,
\end{equation}
and
\begin{equation}\label{EllPrb1e} \mu_2+\calA(x,\partial)\mu_2=f-\diver(D\nabla u)-\calA_1(x,D)\mu_1,\quad x\in\R_+^n;\quad \calB(x,\partial)\mu_2=g,\quad
x\in\pa\R_+^n.
\end{equation}
For $\mu_2$ we have the estimate
$$|\mu_2|_{H_p^2(\R_+^n)}\le
C\left(|u|_{H_p^2(\R_+^n)}+|f|_{L_p(\R_+^n)}+|g|_{W_p^{1-1/p}(\partial\R_+^n)}+|\mu_1|_{H_p^1(\R_+^n)}\right),$$
with some constant $C>0$, since $\calA_1(x,\partial)$ consists solely of
lower order terms with $L_\infty$-coefficients. To obtain the
desired estimate for $\mu$, we therefore have to prove the estimate
$$|\mu_1|_{H_p^1(\R_+^n)}\le C|u|_{H_p^2(\R_+^n)},$$
for the solution $\mu_1$ of \eqref{EllPrb1d}, where $C>0$. For this
purpose note that the $L_p$-realization $A_0$ of $(\calA_0,\calB)$
generates a $C_0$-semigroup in $E_0:=L_p(\R_+^n)$ and $\la+A_0$ is a
linear isomorphism from $E_1:=D(A_0)$ to $E_0$ for each
$\la\in\rho(-A_0)$, the resolvent set of $A_0$. Let
$E_{1/2}:=[E_0,E_1]_{1/2}$, $E_{-1/2}:=(E_{1/2}^\sharp)'$ where
$E^\sharp=E'$ and denote by $A_{-1/2}$ the $E_{-1/2}$-realization of
$A_0$. Here the symbol $[\cdot,\cdot]_{1/2}$ denotes the complex
interpolation functor of exponent $1/2$. It follows from
\cite[Theorem V.2.1.3 \& Corollary V.2.1.4]{Ama} that the operator $A_{-1/2}$ is the
generator of a $C_0$-semigroup with $\rho(A_{-1/2})=\rho(A_0)$ and
$\la+A_{-1/2}:E_{1/2}\to E_{-1/2}$ is a linear isomorphism for each
$\la\in\rho(-A_0)$. It remains to determine the spaces $E_{1/2}$ and
$E_{-1/2}$. To compute $E_{1/2}$, we have to interpolate Sobolev
spaces involving boundary conditions. This has been done e.g. in
\cite{Seeley} and \cite{AmaNonhom}. Following these results it holds
that
$$E_{1/2}=[E_0,E_1]_{1/2}=H_p^1(\R_+^n).$$
Actually in \cite{AmaNonhom} this result was proven for
$C^1$-coefficients but the result remains true for
$W_\infty^1$-coefficients. From this characterization we
obtain
$$E_{-1/2}=\left(H_{p'}^1(\R_+^n)\right)',\quad \frac{1}{p}+\frac{1}{p'}=1,\ 1<p<\infty.$$
Set $F=a\diver(D\nabla u)\in H_p^1(\R_+^n)$ and $f=\diver F\in
L_p(\R_+^n)$. Then $\mu_1\in H_p^2(\R_+^n)$ is a solution of the
abstract equation $\mu_1+A_0\mu_1=f$. We claim that this $f$ can be identified with a linear functional in $E_{-1/2}$ (for which we will write $f$ again).
Indeed the mapping
$$\ph\mapsto\int_{\R_+^n}f\ph\
dx=:\langle f,\ph\rangle_{E_{-1/2},H_{p'}^1},$$ defines a linear
functional on $H_{p'}^1(\R_+^n)$, since by H\"older's inequality
it holds that
\begin{equation*}
|f|_{E_{-1/2}}=\sup_{|\ph|_{H_{p'}^1(\R_+^n)}\le
1}\left|\langle
f,\ph\rangle_{E_{-1/2},H_{p'}^1}\right|=\sup_{|\ph|_{H_{p'}^1(\R_+^n)}\le
1}\left|\int_{\R_+^n} f\ph\ dx\right|\le|f|_{L_p(\R_+^n)}.
\end{equation*}
Integrating by parts we obtain furthermore
\begin{align*}|f|_{E_{-1/2}}&=\sup_{|\ph|_{H_{p'}^1(\R_+^n)}\le
1}\left|\langle
f,\ph\rangle_{E_{-1/2},H_{p'}^1}\right|=\sup_{|\ph|_{H_{p'}^1(\R_+^n)}\le
1}\left|\int_{\R_+^n} f\ph\
dx\right|\\
&=\sup_{|\ph|_{H_{p'}^1(\R_+^n)}\le 1}\left|\int_{\R_+^n} \diver
F\ph\ dx\right|=\sup_{|\ph|_{H_{p'}^1(\R_+^n)}\le
1}\left|\int_{\R_+^n} (F|\nabla\ph)\
dx\right|\\
&\le\sup_{|\ph|_{H_{p'}^1(\R_+^n)}\le
1}|F|_{L_p(\R_+^n;\R^n)}|\ph|_{H_{p'}^1(\R_+^n)}=|F|_{L_p(\R_+^n;\R^n)},
\end{align*}
where we also made use of $(F|\nu)=(a|\nu)\diver(D\nabla u)=0$.
Since $A_{-1/2}$ is the $E_{-1/2}$-realization of $A_0$ (hence an extension of $A_0$) with
$1\in\rho(-A_{-1/2})=\rho(-A_0)$ and since $f=\diver(a\De u)\in
E_{-1/2}$, we obtain a constant $C>0$ such that the estimate
$$|\mu_1|_{H_p^1(\R_+^n)}=|\mu_1|_{E_{1/2}}\le C|f|_{E_{-1/2}}\le C|F|_{L_p(\R_+^n;\R^n)}\le
\tilde{C}|u|_{H_p^2(\R_+^n)},$$ for the solution $\mu_1\in
H_p^2(\R_+^n)$ of \eqref{EllPrb1d} is valid. From the estimates for
$\mu_1$ and $\mu_2$ and the embedding $H_p^2(\R_+^n)\hookrightarrow
H_p^1(\R_+^n)$, we obtain a constant $C>0$ such that
$$|\mu|_{H_p^1(\R_+^n)}\le
C\left(|f|_{L_p(\R_+^n)}+|g|_{H_p^1(\R_+^n)}+|h_1|_{W_p^{1-1/p}(\R^{n-1})}+|u|_{H_p^2(\R_+^n)}\right).$$
In the case that the functions depend on the parameter $t$ it
follows that the estimate
\begin{equation}\label{estmu}|\mu(t)|_{H_p^1(\R_+^n)}\le
C\left(|f(t)|_{L_p(\R_+^n)}+|g(t)|_{H_p^1(\R_+^n)}+|h_1(t)|_{W_p^{1-1/p}(\R^{n-1})}+|u(t)|_{H_p^2(\R_+^n)}\right),\end{equation}
holds for a.e. $t\in J=[0,T]$ where the constant $C>0$ is uniform in
$t$, since the coefficients of the differential operators considered
above are independent of $t$ as well. Taking the $p$-th power and
integrating \eqref{estmu} with respect to $t$, we obtain
\begin{equation}\label{estmu1}|\mu|_{L_p(J;H_p^1(\R_+^n))}\le
C\left(|f|_{L_p(J;(\R_+^n))}+|g|_{L_p(J;H_p^1(\R_+^n))}+|h_1|_{L_p(J;W_p^{1-1/p}(\R^{n-1}))}+|u|_{L_p(J;H_p^2(\R_+^n))}\right).\end{equation}
Finally the estimate for $\pa_t u$ in $L_p(J;L_p(\R_+^n))$ follows
from $\eqref{HS}_2$ and \eqref{estmu1}. The proof is complete.

\bibliographystyle{amsplain}
\bibliography{WilkeBib2}

\end{document}